\documentclass{article}

\usepackage{arxiv}

\usepackage[utf8]{inputenc} 
\usepackage[T1]{fontenc}    
\usepackage{hyperref}       
\usepackage{url}            
\usepackage{booktabs}       
\usepackage{amsfonts}       
\usepackage{nicefrac}       
\usepackage{microtype}      
\usepackage{lipsum}
\usepackage{graphicx}
\graphicspath{ {./images/} }
\usepackage{amsthm, amsfonts, amsmath, amssymb}
\usepackage{multicol}
\usepackage[all,2cell]{xy} \UseAllTwocells \SilentMatrices

\newtheorem{theorem}{Theorem}[section]
\newtheorem{proposition}[theorem]{Proposition}
\newtheorem{corollary}[theorem]{Corollary}
\newtheorem{lemma}[theorem]{Lemma}

\newtheorem{definition}[theorem]{Definition}
\newtheorem{example}[theorem]{Example}
\newtheorem{remark}[theorem]{Remark}

\usepackage[overload]{empheq}
\usepackage{nccmath}

\usepackage{quiver}
\usepackage{tikz-cd}
\usepackage{ dsfont }
\usepackage{mathtools}

\usepackage{tikz}
\usetikzlibrary{matrix}
\usetikzlibrary{calc}
\usepackage{blindtext}
\usepackage{authblk} 
\usepackage{tensor}

\title{Grothendieck prelopologies: towards a closed monoidal sheaf category}

\author[1]{Ana Luiza Tenório}
\author[2]{Hugo Luiz Mariano}
\affil[1]{Department of Applied Mathematics, Federal University of Rio de
Janeiro, \texttt{anatenorio@im.ufrj.br;}}
\affil[3]{Department of Mathematics, University of São Paulo, \texttt{hugomar@ime.usp.br}}

\begin{document}
\maketitle
\begin{abstract}{
In this paper, we present a generalization of Grothendieck pretopologies -- suited for semicartesian categories with equalizers $C$ -- leading to a closed monoidal category of sheaves,  instead of closed cartesian category. This is proved through a different sheafification process, which is the left adjoint functor of the suitable inclusion functor but does not preserve all finite limits. If the monoidal structure in $C$ is given by the categorical product, all constructions coincide with those for Grothendieck toposes. The motivation for such generalization stems from a certain notion of sheaves on quantales that does not form a topos.   }

\end{abstract}


\section{Introduction}\label{sec:intro}

In \cite{luiza2024sheaves} we proposed a new definition of sheaves on semicartesian quantales (a non-idempotent generalization of locales) and showed that the category of sheaves on a quantale $Sh(Q)$ is not necessarily a topos.  This motivates a program towards a monoidal closed notion of elementary and Grothendieck topos. 
In the literature there are versions of sheaves on quantales that form a Grothendieck topos and also are attempts to generalize toposes \cite{coniglio2001modules},\cite{borceux1986quantales}, \cite{borceux1994generic}, \cite{miraglia1998sheaves}, \cite{fourman1979sheaves}, 
 \cite{aguilar2008sheaves}, \cite{heymans2012grothendieck}, \cite{resende2012groupoid}, \cite{mulvey1995quantales}, \cite{gylys2001sheaves}, \cite{coniglio2000non}, \cite{van2007virtual}, \cite{HOHLE199115}.

 On one hand, Van Oystaeyen's work in \cite{van2007virtual} has the approach most similar to ours since is the only one that explicitly changes the notion of a Grothendieck pretopology. Although we are not yet concerned with non-commutative geometry, we are interested in the cohomological aspects of the theory, which we do not develop here but are available at \cite{tenorio2023sheaves}. On the other hand, we are also interested in obtaining a novel notion of toposes that has a linear internal logic. In this perspective,  Coniglio and Miraglia's work in \cite{coniglio2001modules} has a similar spirit to what we would like to achieve in further investigation: their category of sheaves $Sh(Q)$ (different description of sheaves and for idempotent instead of semicartesian quantales) has an internal linear logic and they interpret  Kaplansky's theorem on projective modules in $Sh(Q)$ that translates to showing that every finitely generated projective module over a local ring is free with a finite basis.

Here we start by recalling a few definitions and results we proved in \cite{luiza2024sheaves}. In particular, we say that $\{u_i \in Q\}_{i \in I}$ covers $u$ if $u = \bigvee_{i \in I} u_i$. This is a cover in the sense of a Grothendieck pretopology if the quantale is a locale, but this is not true for quantales in general. Here we develop a generalization of Grothendieck pretopologies that encompass such notion of covering in $Q$, which we call Grothendieck prelopologies\footnote{We expect that sheaves under this cover will have some kind of linear logic as its internal language, which motivated us to choose this terminology.}. The Grothendieck prelopologies are a notion of covering that may be better suited for (semicartesian) monoidal categories with equalizers. If the monoidal structure is given by the categorical product, then Grothendieck prelopologies are precisely the well-known Grothendieck pretopologies. 

After we discuss the Grothendieck prelopologies, we introduce sheaves for them as a functor satisfying certain gluing conditions that can be translated in the format of an equalizer diagram, as usual, but we replace pullbacks by \textit{pseudo-pullbacks}. In this way, we obtain that the category of sheaves for  Grothendieck prelopologies with natural transformations as morphism is a closed monoidal category: the monoidal structure in the base category $C$ induces a monoidal structure in $PSh(C)$ given by the Day convolution. Then we show that the inclusion functor from sheaves for Grothendieck prelopologies to presheaves have a left adjoint functor (called sheafification) that induces the closed monoidal structure in the category of sheaves. The fact that $Sh(Q)$ is not a Grothendieck topos implies that our sheafification cannot preserve all finite limits. We did not explored yet under which circumstances our category of sheaves would be doubly closed monoidal category. 

We discuss the difficulties in defining Grothendieck lopologies and therefore in defining Grothendieck loposes in the usual way. Nevertheless, we conclude these notes by proposing that Grothendieck loposes should be a categorization of quantales in the same way that Grothendieck toposes are a categorization of locales.

We also have an appendix dedicated to study the the projections induced by the semicartesian structures of our categories. The main result in the appendix is Proposition \ref{prop:equalizes}, which may be already known but we could not find a proof. While it is a important result for us, verifying it requires a sequence of small lemmas that could distract from the main topic of the paper. 

We believe that further developments of this theory will be useful both for the categorical logician  interested in a generalization of toposes with a different internal logic and for those interested in sheaf methods in non-commutative geometry, by a careful look at sheaves on non-commutative quantales and simple adaptations in the axioms of the Grothendieck prelopologies. We are also interested in the possible connections with the approach to sheaves via quantaloid enrichment, as explored in \cite{heymans2012grothendieck}.

\textit{Remark: }To avoid size issues, all categories in the domain of a functor are small.

\section{Sheaves on quantales}
First we recall the definition of a locale, which is equivalent to the notion of a complete Heyting Algebra.

\begin{definition}\label{df:locale}
     A \textbf{locale} $(L,\leq)$ is a complete lattice such that 
 
    \begin{center}
      $a \wedge (\bigvee\limits_{i \in I} b_i) = \bigvee\limits_{i \in I}(a \wedge b_i)$, $ \forall a, b_i \in L$.
    \end{center}
\end{definition}

For any locale $L$, viewed as poset category, a presheaf on $L$ is a functor $F: L^{op} \to Set$. If $a \leq b$, we have a function $f: F(b) \to F(a)$, called \textit{restriction map}. Given $t \in F(U)$, we denote by $t_{|_a}$ the application $f(t)$.

\begin{definition}
  A presheaf $F: L^{op} \to Set$ is a \textbf{sheaf on $\mathcal{L}$} if for all $b \in L$ and all $b = \bigvee_{i \in I} b_i$ a cover of $b$ the diagram below is an equalizer in $Set$ 
\begin{center}
     \begin{tikzcd}\label{sheaf}
F (b) \arrow[r, "e"] & \prod\limits_{i\in I}F (b_i) \arrow[r, "p", shift left=1 ex] 
\arrow[r, "q"', shift right=0.5 ex]  & {\prod\limits_{(i,j) \in I \times I}F (b_i \wedge b_j)}
\end{tikzcd}
 \end{center}
 
 where:
 \begin{enumerate}
     \item $e(t) = \{t_{|_{b_i}} \enspace | \enspace i \in I\}, \enspace t \in F (b)$ 
     \item     $p((t_k)_{ k \in I}) = (t_{i_{|_{b_i \wedge b_j}}})_{(i,j)\in I\times I}$ \\ $q((t_k)_{k \in I}) = (t_{j_{|_{b_i \wedge b_j}}})_{(i,j)\in I\times I}, \enspace (t_k)_{k \in I} \in \prod\limits_{k\in I}F (b_k)$
 \end{enumerate}\label{df:sheaf_locale} 
\end{definition}

A morphism between sheaves is a natural transformation and then we obtain the category of sheaves on $L$, denoted by $Sh(L)$. The definition of sheaves on locales is well-established and the most famous case is for the locale of open subsets of a topological space, where the arbitrary join is the union and the binary meet is the intersection.

A quantale is a generalization of the notion of a locale.
\begin{definition}\label{df:quantale}
A \textbf{quantale} $Q$ is a complete lattice structure  $(Q, \leq)$ with an associative  binary operation $\odot: {Q} \times {Q} \to {Q}$ (called multiplication) such that for all $a\in  Q$ and $\{b_i\}_{i\in I} \subseteq Q$ the following distributive laws hold:
\begin{enumerate}
    \item $a \odot (\bigvee\limits_{i \in I}b_i) = \bigvee\limits_{i \in I}(a \odot b_i)$
    \item $(\bigvee\limits_{i \in I}b_i)  \odot a = \bigvee\limits_{i \in I}(b_i \odot a)$
\end{enumerate}
\end{definition}
If $\odot$ is the infimum operation then the quantale is a locale. 

\begin{definition}\label{df:quantalewithproperties}
A quantale $Q = (Q, \leq, \odot)$ is
\begin{enumerate}
    \item \textbf{commutative} when $(Q, \odot)$ is a commutative semigroup;
    \item  \textbf{idempotent} when $a \odot a = a$, for $a \in Q$;
    \item \textbf{right-sided} when $a \odot \top = a$, for all $a \in Q$, where $\top$ is the top member of the poset;
    \item \textbf{semicartesian}  when $a \odot b \leq a, b$, for all $a,b \in Q$;
    \item \textbf{integral} when $Q$ is unital and $1 = \top$.
\end{enumerate}
\end{definition}

 For our definition of a sheaf on a quantale we will require semicartesianity and for simplicity we will also assume that the quantale is unital. Some examples of such kind of quantales are $(\mathds{N}\cup \{\infty\}, \geq, +)$, $([0,\infty], \geq, +)$, and $(\mathcal{I}(R), \subseteq, . )$, where $\mathcal{I}(R)$ is the set  of ideals of a commutative and unital ring $R$ and the multiplication is the multiplication of ideals (the supremum is given by the sum of ideals).

\begin{remark} \label{quantale-re} Note that:
\begin{enumerate}
\item A quantale $(Q, \leq, \odot)$ is semicartesian iff $a\odot b \leq a \wedge b$, for all $a, b \in Q$.
\item Let $Q$ be a unital quantale, then it is integral iff it  is semicartesian.  Indeed: suppose that $Q$ is integral, since $b \leq \top$ we have $a \odot b \leq a \odot \top = a \odot 1 = a$, then $Q$ is semicartesian; conversely, suppose that $Q$ is semicartesian, since $\top = \top \odot 1 \leq 1$, then $\top = 1$.    
\end{enumerate}
\end{remark}

In locales, arbitrary sups distribute over infimun while in quantales arbitrary sups distribute over the multiplication. Therefore, we replace $\wedge$ by $\odot$ for the definition of sheaves on quantales.
\begin{definition}\label{df:sheaf-on-quantales}
  A presheaf $F:Q^{op} \to Set$ is a \textbf{sheaf on $Q$} when for all $u \in Q$ and all $u = \bigvee_{i \in I} u_i$ cover of $u$ the following diagram is an equalizer in $Set$
\begin{center}
     \begin{tikzcd}
F (u) \arrow[r, "e"] & \prod\limits_{i\in I}F (u_i) \arrow[r, "p", shift left=1 ex] 
\arrow[r, "q"', shift right=0.5 ex]  & {\prod\limits_{(i,j) \in I \times I}F (u_i \odot u_j)}
\end{tikzcd}
 \end{center}
  where:
 \begin{enumerate}
     \item $e(t) = \{t_{|_{u_i}} \enspace | \enspace i \in I\}, \enspace t \in F (u)$ 
     \item     $p((t_k)_{ k \in I}) = (t_{i_{|_{u_i \odot u_j}}})_{(i,j)\in I\times I}$ \\ $q((t_k)_{k \in I}) = (t_{j_{|_{u_i \odot u_j}}})_{(i,j)\in I\times I}, \enspace (t_k)_{k \in I} \in \prod\limits_{k\in I}F (u_k)$
 \end{enumerate}
\end{definition}

Morphisms between sheaves on quantales are natural transformations and we denote the category of sheaves on quantales by $Sh(Q)$. Notice that the maps $F(u_i) \to F(u_i \odot u_j)$ exist because $u_i \odot u_j $ always is less or equal to $u_i$ and $u_j$, for all $i,j\in I$.

In \cite{luiza2024sheaves}, we show that sheaves on quantales and sheaves on locales behave similarly at the same time that, in general, the category $Sh(Q)$ is not a topos. Therefore, what kind of category is $Sh(Q)$?  We will discuss this in the following sections. The main problem that we solve consists of generalizing Grothendieck pretolopogies: note that $\bigvee_{i \in I} u_i = u$ in a locale is a covering in the sense of a Grothendieck pretopology but this is not true for quantales because, in general, the stability under pullback does not hold. 

\section{Grothendieck prelopologies}\label{sec:Grothendieckprelopologies}

First, we recall the definition of a Grothendieck pretopology.
\begin{definition}\label{df: grothendieck pretop} 
   Let $\mathcal{C}$ be a small category with finite limits (or just with pullbacks). A \textbf{Grothendieck pretopology} on $\mathcal{C}$ associates to each object $U$ of $\mathcal{C}$ a set $K(U)$ of families of morphisms $\{U_i \rightarrow U\}_{i \in I}$ satisfying some rules:
    \begin{enumerate}
        \item The singleton family $\{U' \xrightarrow{f} U \}$, formed by an isomorphism $f : U' \overset{\cong}\to U$,  is in $K(U)$;
                
        \item If $\{U_i \xrightarrow{f_i} U\}_{i \in I}$ is in $K(U)$ and $\{V_{ij} \xrightarrow{g_{ij}} U_i \}_{j \in J_i}$ is in $K(U_i)$ for all $i \in I$, then $\{V_{ij} \xrightarrow{f_i \circ g_{ij}} U \}_{i \in I, j \in J_i}$ is in $K(U)$; 
        \item If $\{U_i \rightarrow U\}_{i \in I}$ is in $K(U)$, and $V \rightarrow U$ is any  morphism in $\mathcal{C}$, then the family of pullbacks $\{V \times_U U_i \rightarrow V \}$ is in $K(V)$.
    \end{enumerate}
 \end{definition}
    Some authors also call the above of \textbf{a basis for a Grothendieck topology}.

If our category is given by a locale, then the pullback is the infimum. Obverse that $u = \bigvee_{i \in I} u_i$ gives a Grothendieck pretopology. We  check the third axiom (stability under pullback) since it is the most relevant one to us. Take $v \leq u$, then $$\bigvee_{i \in I}(v \wedge u_i) = v \wedge (\bigvee_{i \in I} u_i) = v \wedge u = v. $$
Thus, $v \wedge u_i$ cover $v$, as desired.

So the distributivity plays an essential role and we want to keep it when dealing with quantales. However, the multiplication $\odot$ is not the pullback in $Q$. In the next definition we aim to generalize the notion of a pullback to encompass the quantalic multiplication. Since our quantales are semicartesian, we also want \textit{semicartesian categories}, i.e., a monoidal $(\mathcal{C},\otimes,1)$ category in which the unit $1$ also is the terminal object. We use $!_X$ to denote the unique morphism from an object $X$ to the terminal object.

\begin{definition}
Let $(\mathcal{C},\otimes,1)$ be a semicartesian monoidal category with equalizers. The \textbf{pseudo-pullback} of morphisms $f: A \to C$ and $g: B \to C$ is the equalizer of the parallel arrows \begin{tikzcd}
	{A \otimes B} & C
	\arrow["{g\circ \pi_2}"', shift right=1, from=1-1, to=1-2]
	\arrow["{f \circ \pi_1}", shift left=1, from=1-1, to=1-2]
\end{tikzcd} where $\pi_1 = \rho_A \circ (id_A \otimes !_B)$ and  $\pi_2 = \lambda_B \circ (!_A \otimes id_B)$
\end{definition}

Diagrammatically, the pseudo-pullback is the object $A\tensor[_{f}]{\otimes}{_{g}}B $ with the equalizer arrow in the following equalizer diagram

\[\begin{tikzcd}
	{A\tensor[_{f}]{\otimes}{_{g}}B} \\
	& {A\otimes B} & B \\
	& A & C
	\arrow["g", from=2-3, to=3-3]
	\arrow["f"', from=3-2, to=3-3]
	\arrow["{\pi_2}", from=2-2, to=2-3]
	\arrow["{\pi_1}"', from=2-2, to=3-2]
	\arrow["e'", from=1-1, to=2-2]
	\arrow["{p_1}"', curve={height=6pt}, from=1-1, to=3-2]
	\arrow["{p_2}", curve={height=-12pt}, from=1-1, to=2-3]
\end{tikzcd}\]

The arrows $p_1$ and $p_2$ are just the compositions $\pi^1 \circ e'$ and $\pi_2 \circ e'$, respectively. When needed, we will indicate the domain of the projections using $p^j_{A,B}$ and $\pi^j_{A,B}$, where $j = 1,2$, for example.

\begin{example}\label{exa:pb_is_pseudopullback}
    If the monoidal tensor is the categorical product, then the pseudo-pullback is the pullback.
\end{example}
\begin{example}\label{exa:psedo-pb_is_multiplication}
If $(\mathcal{C},\otimes,1) = (Q, \odot , 1)$, then the pseudo-pullback is the multiplication $\odot$.
\end{example}

\begin{example}
    Let $(\mathcal{C},\otimes,I)$ be a monoidal category. The slice category $\mathcal{C}/I$ is a semicartesian monoidal category with tensor defined by $$(A \xrightarrow{\phi} I)\otimes_{\mathcal{C}/I}(B\xrightarrow{\psi} I) = (A\otimes B \xrightarrow{\phi \otimes \psi} I\otimes I \cong I). $$

    If 
\begin{tikzcd}[ampersand replacement=\&]
	E \& A \& B
	\arrow["q"', shift right=1, from=1-2, to=1-3]
	\arrow["p", shift left=1, from=1-2, to=1-3]
	\arrow["e", from=1-1, to=1-2]
\end{tikzcd} is an equalizer diagram in $(\mathcal{C},\otimes,I)$  then 
\[\begin{tikzcd}[ampersand replacement=\&]
	E \& A \& B \\
	\& I
	\arrow["q"', shift right=1, from=1-2, to=1-3]
	\arrow["p", shift left=1, from=1-2, to=1-3]
	\arrow["e", from=1-1, to=1-2]
	\arrow["\phi"', from=1-2, to=2-2]
	\arrow["\psi", from=1-3, to=2-2]
	\arrow["{\phi\circ e}"', from=1-1, to=2-2]
\end{tikzcd}\]
is an equalizer diagram in $(\mathcal{C}/I,\otimes_{\mathcal{C}/I},I)$.
    In particular, suppose that $(\mathcal{C},\otimes,I)$ is the category of $R$-modules, where $R$ is a commutative ring with unity, the monoidal structure is given by the tensor product of $R$-modules and the monoidal unity is the ring $R$. Then the objects in $(\mathcal{C}/I,\otimes_{\mathcal{C}/I},I)$ are module homomorphisms $M\to R$, where $M$ is an $R$-module. Now, $R$ is not only the monoidal unity but also the terminal object in $(\mathcal{C}/I,\otimes_{\mathcal{C}/I},I)$. So we have ``projections'' $\pi_1:M \otimes_R N \to M \otimes_R R \to M$ and  $\pi_2:M \otimes_R N \to R \otimes_R N \to N$. Then the pseudo-pullback of $f:M \to O$ and $g: N \to O$ is obtained from the equalizer of 
\begin{tikzcd}[ampersand replacement=\&]
	{M\otimes_R N} \& O
	\arrow["{g\circ \pi_2}", shift left=1, from=1-1, to=1-2]
	\arrow["{f \circ\pi_1}"', shift right=1, from=1-1, to=1-2]
\end{tikzcd} in the category of $R$-modules and module homomorphisms.
    
\end{example}

\begin{example}\label{exa:ppb_product_category}
    Consider two categories $(\mathcal{C},\otimes,1_{\mathcal{C}})$ and $(\mathcal{D},\star,1_{\mathcal{D}})$ both with pseudo-pullbacks. Then the product category $\mathcal{C} \times \mathcal{D}$ has pseudo-pullbacks: given $C, C'$ objects in $\mathcal{C}$ and $D, D'$ objects in $\mathcal{D}$. Define $$(C,D) \otimes_{\mathcal{C}\times \mathcal{D}} (C',D') = (C\otimes C', D \star D')$$
    If $C\tensor[_{f}]{\otimes}{_g}C'$ is the pseudo-pullback of $f: C \to E$ and $g: C' \to E$, and  $D\tensor[_{\phi}]{\star}{_{\psi}}D'$ is the pseudo-pullback of $\phi : D \to F$ and $\psi: D' \to F$, then the pseudo-pullback of  $(f,\phi): (C,D) \to (E,F)$ and $(g,\psi):(C',D')\to (E,F)$ is $(C\tensor[_{f}]{\otimes}{_{g}}C',D\tensor[_{\phi}]{\star}{_{\psi}}D')$. In particular, we can take $(\mathcal{C},\otimes,1_{\mathcal{C}}) = (\mathcal{C},\times, 1)$ and $(\mathcal{D},\star,1_{\mathcal{D}}) = (Q,\odot,1)$.  
\end{example}

The reader may think that our generalization of Grothendieck pretopologies requires a stability under pseudo-pullbacks, but this is not enough since if $u = \bigvee_{i \in I}u_i$ and $v \leq u$ then
$$\bigvee_{i \in I} v \odot u_i = v \odot \bigvee_{i \in I} u_i = v \odot u \leq v$$

However, for any $v \in Q$ we have that $\bigvee_{i \in I} v \odot u_i = v \odot \bigvee_{i \in I} u_i = v \odot u$. So we define:

\begin{definition}
         Let $(\mathcal{C},\otimes,1)$ be a  monoidal category. A \textbf{weak Grothendieck prelopology} on $\mathcal{C}$ associates to each object $U$ of $\mathcal{C}$ a set $L(U)$ of families of morphisms $\{U_i \rightarrow U \}_{i \in I}$ such that:
        \begin{enumerate}
        \item The singleton family $\{U' \xrightarrow{f} U \}$, formed by an isomorphism $f : U' \overset{\cong}\to U$,  is in $L(U)$;
        
        \item If $\{U_i \xrightarrow{f_i} U\}_{i \in I}$ is in $L(U)$ and $\{V_{ij} \xrightarrow{g_{ij}} U_i \}_{j \in J_i}$ is in $L(U_i)$ for all $i \in I$, then $\{V_{ij} \xrightarrow{f_i \circ g_{ij}} U \}_{i \in I, j \in J_i}$ is in $L(U)$; 
        \item If $\{ f_i : U_i \to U\}_{i \in I} \in L(U)$, then $\{ f_i \otimes id_V : U_i \otimes V \to U \otimes V\}_{i \in I}$ is in $L(U\otimes V)$ and $\{ id_V \otimes f_i : V \otimes U_i \to V \otimes U\}_{i \in I}$ is in $L(V\otimes U)$, for any $V$ object in $\mathcal{C}$.
        \end{enumerate}
\end{definition}

Since we are not using the pseudo-pullback yet, $(\mathcal{C},\otimes,1)$ does not need to be semicartesian.
\begin{example}
    \begin{enumerate}
        \item Let $R$ be a commutative ring with unity. The category of $R$-algebras is monoidal with the tensor product of algebras as the product and $R$ as the unit. Consider $M$ an $R$-module and $I$ an index set. Define $$\{M_i \to M \}_{i \in I} \in L(M) \iff M \cong \oplus_{i \in I}M_i$$

        The first axiom clearly holds. The second is also simple: if $M = \oplus_{i \in I}M_i$ and $M_i = \oplus_{j \in J_i} M_{ij}$ for all $i \in I$, then $M \cong \oplus_{i \in I}\oplus_{j \in J} M_{ij} \cong  \oplus_{i \in I, j \in J_i} M_{ij}$. 

        The third axiom holds because the tensor product distributes over direct sum, i.e., if $N$ in a $R$-module, that is an isomorphism $N\otimes(\oplus_{i\in I}M_i)\cong \oplus_{i\in I}(N\otimes M_i)$.  

        \item In $Top$, the category of topological spaces with continuous functions, define:
        \[
        \{X_i \to X\} \in L(X) \iff X \approx \coprod X_i,
        \]
        where $\coprod X_i$ is the disjoint union of topological spaces $X_i$.  Similarly to the above case we have $X \approx \coprod_{i \in I, j \in J_i} X_{ij}$ and $Y\times(\coprod_{i\in I}X_i)\approx \coprod_{i\in I}(Y\times X_i)$

        \item In $Top_{\{*\}}$, the category of pointed  topological spaces, also has a similar behavior. Replace the product of topological spaces with the smash product and the  disjoint union with the wedge sum.
    \end{enumerate}
\end{example}

As we are going to see, if $\mathcal{C}$ is a category with pullbacks and $K$ is Grothendieck pretopology, then $K$ is a (weak) Grothendieck prelopology. However, we want to also obtain that if $K$ is a weak Grothendieck prelopology and the monoidal tensor is the categorical product, then $K$ is a Grothendieck pretopology. This does not hold for weak prelopologies but does hold if we add another axiom.

\begin{definition}
        Let $(\mathcal{C},\otimes,1)$ be a semicartesian monoidal category with pseudo-pullbacks. A \textbf{Grothendieck prelopology} on $\mathcal{C}$ associates to each object $U$ of $\mathcal{C}$ a set $L(U)$ of families of morphisms $\{U_i \rightarrow U \}_{i \in I}$ such that:
        \begin{enumerate}
        \item The singleton family $\{U' \xrightarrow{f} U \}$, formed by an isomorphism $f : U' \overset{\cong}\to U$,  is in $L(U)$;
        
        \item If $\{U_i \xrightarrow{f_i} U\}_{i \in I}$ is in $L(U)$ and $\{V_{ij} \xrightarrow{g_{ij}} U_i \}_{j \in J_i}$ is in $L(U_i)$ for all $i \in I$, then $\{V_{ij} \xrightarrow{f_i \circ g_{ij}} U \}_{i \in I, j \in J_i}$ is in $L(U)$; 
        \item If $\{ f_i : U_i \to U\}_{i \in I} \in L(U)$, then $\{ f_i \otimes id_V : U_i \otimes V \to U \otimes V\}_{i \in I}$ is in $L(U\otimes V)$ and $\{ id_V \otimes f_i : V \otimes U_i \to V \otimes U\}_{i \in I}$ is in $L(V\otimes U)$, for any $V$ object in $\mathcal{C}$;
        \item If $\{U_i \xrightarrow{f_i} U\}_{i \in I}$ is in $L(U)$ and $g: V \to U$ is any morphism in $\mathcal{C}$, then $\{\phi_i : U_i\tensor[_{f_i}]{\otimes}{_{g}}V \to U\tensor[_{id_U}]{\otimes}{_{g}}V\}_{i\in I}$ is in $ L(U\tensor[_{id_U}]{\otimes}{_{g}}V)$ and $\{\phi_i : V\tensor[_{g}]{\otimes}{_{f_i}}U \to V\tensor[_{g}]{\otimes}{_{id_U}}U\}_{i\in I}$ is in $ L(V\tensor[_{g}]{\otimes}{_{id_U}}U)$.
        \end{enumerate}
\end{definition}\label{df:prelopolgy}
\begin{remark}\label{diagram:axiom4}
For each $i \in I$, the arrow $\phi_i : U_i\tensor[_{f_i}]{\otimes}{_{g}}V \to U\tensor[_{id_U}]{\otimes}{_{g}}V$ is unique because of the universal property of the equalizer:
\[\begin{tikzcd}
	{U_i\tensor[_{f_i}]{\otimes}{_{g}}V} && {U_i \otimes V} && {U_i} \\
	{U\tensor[_{id_U}]{\otimes}{_{g}}V} && {U \otimes V} && U \\
	&&& V
	\arrow["{f_i}", from=1-5, to=2-5]
	\arrow["{\pi^i_1}", from=1-3, to=1-5]
	\arrow["{\pi_1}", from=2-3, to=2-5]
	\arrow["{f_i \otimes id_V}"', from=1-3, to=2-3]
	\arrow["{e_i}", tail, from=1-1, to=1-3]
	\arrow["{\pi_2}"', from=2-3, to=3-4]
	\arrow["g"', from=3-4, to=2-5]
	\arrow["e"', tail, from=2-1, to=2-3]
	\arrow["{\phi_i}"', dashed, from=1-1, to=2-1]
\end{tikzcd}\]
\end{remark}

For noncommutative geometry, it may be interesting to study a version of this definition where only one side of the third and the fourth axioms hold.  

\begin{example}\label{ex:prelopology_in_quantales}
If $(\mathcal{C}, \otimes, 1)$ is a semicartesian quantale $(Q, \odot, 1)$, then $$\{f_i: U_i \to U\} \in L(U) \iff U = \bigvee_{i\in I}U_i$$ determines a Grothendieck prelopology: Axioms $1$ and $2$ are immediate, axiom $3$ was proved in the previous discussion. The last axiom follows from observing that $U_i\tensor[_{f_i}]{\otimes}{_{g}}V = U_i \otimes V$ (see \ref{exa:psedo-pb_is_multiplication} ) and then realizing that $\phi_i = f_i \otimes id_V$ for all $i \in I$. In other words, in the quantalic case the third axiom, implies the fourth one.

Of course, if $(Q, \odot, 1)$ is a locale, then the above prelopology coincides with the usual Grothendieck pretopology in $Q$ --  because $v\leq u$ iff $v \wedge u = v= u \wedge v $.
\end{example}

Now, note that if the tensor product is given by the cartesian product then we have a diagram of the following form
\[\begin{tikzcd}[ampersand replacement=\&]
	{U_i\times_U V} \&\& {U_i \times V} \&\& {U_i} \\
	{U\times_U V} \&\& {U \times V} \&\& U \\
	\&\&\& V
	\arrow["{f_i}", from=1-5, to=2-5]
	\arrow["{\pi^i_1}", from=1-3, to=1-5]
	\arrow["{\pi_1}", from=2-3, to=2-5]
	\arrow["{f_i \times id_V}"', from=1-3, to=2-3]
	\arrow["{e_i}", tail, from=1-1, to=1-3]
	\arrow["{\pi_2}"', from=2-3, to=3-4]
	\arrow["g"', from=3-4, to=2-5]
	\arrow["e"', tail, from=2-1, to=2-3]
	\arrow["{\phi_i}"', dashed, from=1-1, to=2-1]
\end{tikzcd}\]

In Proposition \ref{prop:prelopology_in_cartesiancategory_is_pretology}, we prove that in such case the fourth axiom is the stability (under pullbacks) axiom of the definition of a Grothendieck pretopology. Since, the stability under pullback implies that $\{U_i \times V \xrightarrow{f \times id_V} U \times V\}$ cover $U\times V$ (proved in Proposition \ref{pro:everypretopology_is_a_prelopology} bellow) we can say that, in the cartesian case, the fourth axiom implies the third one.

So the following proposition shows that we are generalizing Grothendieck pretopologies.

\begin{proposition}\label{pro:everypretopology_is_a_prelopology}
If $\mathcal{C}$ is a cartesian category with equalizers and $K$ is a Grothendieck pretopology in $\mathcal{C}$, then $K$ is a Grothendieck prelopology.
\end{proposition}
\begin{proof}

In this case, the pseudo-pullback is the pullback. The first two axioms are automatically satisfied. It remains to prove the last two. Let $\{f_i : U_i \to U\} \in K(U)$

$3.$ It holds because  
\[\begin{tikzcd}
	{U_i\times V} & {U \times V} \\
	U_i &  U
	\arrow["{\pi_1}", from=1-2, to=2-2]
	\arrow["f_i"', from=2-1, to=2-2]
	\arrow["{f_i \times id_V}", from=1-1, to=1-2]
	\arrow["{\pi^i_1}"', from=1-1, to=2-1]
\end{tikzcd}\]
is a pullback diagram.

$4.$ Observe the following diagram, where we use the same notation of Remark \ref{diagram:axiom4}. 
\[\begin{tikzcd}
	{U_i \times_U V} & {U_i} \\
	{U\times_U V} & U \\
	V & U
	\arrow["{\phi_i}"', from=1-1, to=2-1]
	\arrow["{\pi^1_i \circ e_i}", from=1-1, to=1-2]
	\arrow["{f_i}", from=1-2, to=2-2]
	\arrow["{\pi^1\circ e}", from=2-1, to=2-2]
	\arrow["{id_U}", from=2-2, to=3-2]
	\arrow["{\pi_2 \circ e}"', from=2-1, to=3-1]
	\arrow["g"', from=3-1, to=3-2]
\end{tikzcd}\]

The outer rectangle is a pullback, and the square on the bottom is a pullback. So, by the pullback lemma, the square on the top is a pullback. Since $K$ is a Grothendieck topology, $\phi_i  \in L(U \times_U V)$. 

\end{proof}

Conversely, 

\begin{proposition}\label{prop:prelopology_in_cartesiancategory_is_pretology}
    If $L$ is a Grothendieck prelopology on a cartesian category with equalizers $\mathcal{C}$, then $L$ is a Grothendieck pretopology.
\end{proposition}
\begin{proof}
  The only axiom that we have to prove is the stability (under pullback) axiom. Consider $\{f_i: U_i \to U\} \in L(U)$ and $g: V \to U$ a morphism in $\mathcal{C}$. Using the same notion as in Definition \ref{df:prelopolgy}, we know that $\{\phi_i : U_i\tensor[_{f_i}]{\otimes}{_{g}}V \to U\tensor[_{id_U}]{\otimes}{_{g}}V\}_{i\in I} \in L(U\tensor[_{id_U}]{\otimes}{_{g}}V)$. Since $ \otimes  =\times $, the pseudo-pullback $U_i\tensor[_{f_i}]{\otimes}{_{g}}V$ is the pullback $U_i \times_U V$ and $U\tensor[_{id_U}]{\otimes}{_{g}}V$ is the pullback $U \times_U V$. Now, note that the following is a pullback diagram 
\[\begin{tikzcd}[ampersand replacement=\&]
	{U\times_U V} \& V \\
	U \& U
	\arrow["g", from=1-2, to=2-2]
	\arrow["{\pi_2 \circ e}", from=1-1, to=1-2]
	\arrow["{\pi_1 \circ e}"', from=1-1, to=2-1]
	\arrow["{id_U}"', from=2-1, to=2-2]
\end{tikzcd}\]
Then $U\times_U V \to V$ is an isomorphism, which implies that $\{\pi_2 \circ e: U\times_U V \to V\} \in L(V)$, by the first axiom of a Grothendieck prelopology. Thus, the composition $\{U_i \times_U V \xrightarrow{\phi_i} U \times_U V \xrightarrow{e} U\times V \xrightarrow{\pi_2} V\}_{i\in I} \in L(V)$, proving that $L$ satisfies the stability axiom. 
\end{proof}


\begin{example}\label{ex:prelopology_product_cat}
    \textbf{Prelopology for product category:} Consider two categories $(\mathcal{C},\otimes,1_{\mathcal{C}})$ and $(\mathcal{D},\star,1_{\mathcal{D}})$ both with pseudo-pullbacks and equipped, respectively, with Grothendieck prelopologies $L_{\mathcal{C}}$ and $L_{\mathcal{D}}$. Define a Grothendieck prelopology in $\mathcal{C}\times \mathcal{D}$ by $\{(\gamma_i,\delta_i):(C_i,D_i)\to (C,D)\} \in L_{\mathcal{C}\times \mathcal{D}}(C,D)$ iff $\{\gamma_i : C_i \to C \} \in L_{\mathcal{C}}(C) \mbox{ and } \{\delta_i : D_i \to D \} \in L_{\mathcal{D}}(D).$

    The verification is straightforward and we are going to show the calculations only for the fourth axiom. Since $\{\gamma_i : C_i \to C \} \in L_{\mathcal{C}}(C)$, for any $\epsilon : E \to C $ in $\mathcal{C}$ we have
\[\begin{tikzcd}[ampersand replacement=\&]
	{C_i\tensor[_{\gamma_i}]{\otimes}{_{\epsilon}}E} \&\& {C_i \otimes E} \&\& {C_i} \\
	{C\tensor[_{id_C}]{\otimes}{_{\epsilon}}E} \&\& {C \otimes E} \&\& C \\
	\&\&\& E
	\arrow["{\gamma_i}", from=1-5, to=2-5]
	\arrow["{\pi^i_1}", from=1-3, to=1-5]
	\arrow["{\pi_1}", from=2-3, to=2-5]
	\arrow["{\gamma_i \otimes id_E}"', from=1-3, to=2-3]
	\arrow["{e_i}", tail, from=1-1, to=1-3]
	\arrow["{\pi_2}"', from=2-3, to=3-4]
	\arrow["\epsilon"', from=3-4, to=2-5]
	\arrow["e"', tail, from=2-1, to=2-3]
	\arrow["{\phi_i}"', dashed, from=1-1, to=2-1]
\end{tikzcd}\]
with $\{C_i\tensor[_{\gamma_i}]{\otimes}{_{\epsilon}}E \to C\tensor[_{id_C}]{\otimes}{_{\epsilon}}E \}_{i\in I} \in L_{\mathcal{C}}(C\tensor[_{id_C}]{\otimes}{_{\epsilon}}E)$.
Analogously, for any $\zeta : F \to D $ in $\mathcal{D}$ we have 
\[\begin{tikzcd}[ampersand replacement=\&]
	{D_i\tensor[_{\gamma_i}]{\star}{_{\epsilon}}F} \&\& {D_i \star F} \&\& {D_i} \\
	{D\tensor[_{id_D}]{\star}{_{\zeta}}F} \&\& {D \star F} \&\& D \\
	\&\&\& F
	\arrow["{\pi^i_1}", from=1-3, to=1-5]
	\arrow["{\pi_1}", from=2-3, to=2-5]
	\arrow["{\delta_i \otimes id_F}"', from=1-3, to=2-3]
	\arrow["{e'_i}", tail, from=1-1, to=1-3]
	\arrow["{\pi_2}"', from=2-3, to=3-4]
	\arrow["\zeta"', from=3-4, to=2-5]
	\arrow["{e'}"', tail, from=2-1, to=2-3]
	\arrow["{\phi'_i}"', dashed, from=1-1, to=2-1]
	\arrow["{\delta_i}", from=1-5, to=2-5]
\end{tikzcd}\]
with $\{D_i\star[_{\gamma_i}]{\otimes}{_{\epsilon}}F \to D\tensor[_{id_D}]{\star}{_{\zeta}}F\}_{i\in I}\in L_{\mathcal{D}}(D\tensor[_{id_D}]{\otimes}{_{\zeta}}F)$
By Example \ref{exa:ppb_product_category}, for each $i \in I$, the pseudo-pullback of $(\gamma_i,\delta_i): (C_i,D_i) \to (C,D)$ and $(\epsilon,\zeta):(E,F)\to (C,D)$ is $(C_i\tensor[_{\gamma_i}]{\otimes}{_{\epsilon}}E,D_i\tensor[_{\delta_i}]{\star}{_{\zeta}}F)$. Thus, $\{(C_i\tensor[_{\gamma_i}]{\otimes}{_{\epsilon}}E,D_i\tensor[_{\delta_i}]{\star}{_{\zeta}}F)\to (C\tensor[_{id_C}]{\otimes}{_{\epsilon}}E, D\tensor[_{id_D}]{\star}{_{\zeta}}F) \} \in L_{\mathcal{C}\times \mathcal{D}}((C\tensor[_{id_C}]{\otimes}{_{\epsilon}}E, D\tensor[_{id_D}]{\star}{_{\zeta}}F)).$
\end{example}

\begin{remark}
    Observe that if $L_{\mathcal{C}}$ is a Grothendieck pretopology and $L_{\mathcal{D}}$ is a quantalic covering, this construction provides an example of prelopology that is not quantalic neither is a Grothendieck pretopology.
\end{remark}

Next, mimicking the definition of a sheaf for a Grothendieck pretopology we define

\begin{definition}\label{grothsheaf}
 Let $\mathcal{C} = (\mathcal{C}, \otimes,1)$ be a monoidal semicartesian category with equalizers. A presheaf $F: \mathcal{C}^{op} \rightarrow Set$ is \textbf{a sheaf for the Grothendieck prelopology} $L(U) = \{f_i : U_i \to U\}_{i \in I}$ if the following diagram is an equalizer in $Set$:  

             \begin{center}
              \begin{tikzcd}
            F (U) \arrow[r,"e"] & \prod\limits_{i \in I} F (U_i) \arrow[r,"p", shift left=1 ex] 
            \arrow[r, "q"', shift right=0.5 ex]  & \prod\limits_{(i,j)\in I\times I} F (U_i {}_{f_i}\!\!\otimes_{f_j} U_j)
            \end{tikzcd}
            \end{center} 
where 
 \begin{enumerate}
     \item $e(f) = \{F(f_i)(t) \enspace | \enspace i \in I\}, \enspace f \in F (U)$ 
     \item     $p((f_k)_{ k \in I}) = (F(p^1_{ij})(f_i))_{(i,j)\in I\times I}$ \\ $q((f_k)_{k \in I}) = (F(p^2_{ij})(f_j))_{(i,j)\in I\times I}, \enspace (f_k)_{k \in I} \in \prod\limits_{k\in I}F (U_k)$
 \end{enumerate}
 with $p^1_{i,j} = e' \circ \pi^1_{i,j}$ and $p^2_{i,j} = e' \circ \pi^2_{i,j}$
\end{definition}\label{df:sheaf_on_prelopology}

As in the classic definition, it is possible to define sheaves using the existence of a unique gluing for compatible families.

\begin{definition}
    Let $L$ be a Grothendieck prelopology, $\{f_i: U_i \to U\}_{i \in I}\in L(U)$, and $F$ a presheaf. We say that a family $\{x_{f_i} \in F(U_i)\}_{i \in I}$ is \textbf{compatible} for $\{f_i: U_i \to U\}$ if $F(p^1_{i,j})(x_{f_i}) = F(p^2_{i,j})(x_{f_j})$ for all $i, j \in I$, where $p^1_{i,j}$ and $p^2_{i,j}$ are the projections on the pseudo-pullback.
\end{definition}
\begin{definition}
    A presheaf $F$ is a \textbf{sheaf} if for every cover $\{f_i: U_i \to U\}_{i \in I}\in L(U)$, and every compatible family $\{x_{f_i} \in F(U_i)\}_{i \in I}$, there is a unique $x \in F(U)$ such that $F(f_i)(x) = x_{f_i}$ for all $i \in I$.
\end{definition}

It is clear that those two definitions of a sheaf are equivalent.

\begin{example}
\textbf{Sheaves on quantales:} In this case, $(\mathcal{C}, \otimes, 1) = (Q,\odot, 1)$ and the Grothendieck prelopology is given by $\{f_i: U_i \to U\} \in L(U) \iff U = \bigvee_{i\in I}U_i$ (see Example \ref{ex:prelopology_in_quantales}). Since in $Q$ the pseudo-pullback is given by the quantalic product, the definition of sheaves on quantales (\ref{df:sheaf-on-quantales}) fits perfectly  in the above definition.
\end{example}

\begin{example}\label{ex:sheaves_product_cat}
    \textbf{Sheaves on the product category:}  Consider two categories $(\mathcal{C},\otimes,1_{\mathcal{C}})$ and $(\mathcal{D},\star,1_{\mathcal{D}})$ both with pseudo-pullbacks and equipped, with Grothendieck prelopologies $L_{\mathcal{C}}$ and $L_{\mathcal{D}}$, respectively. Let $F:\mathcal{C}^{op} \to Set$ be a sheaf for $L_{\mathcal{C}}$ and $G:\mathcal{D}^{op} \to Set$ be a sheaf for $L_{\mathcal{D}}$.  In Example \ref{ex:prelopology_product_cat} we described a Grothendieck prelopology $L_{\mathcal{C}\times \mathcal{D}}$ for the product category $\mathcal{C}\otimes \mathcal{D}$. Then $F\times G: (\mathcal{C}\times \mathcal{D})^{op} \to Set$ defined by $(F\times G)((C,D)) = F(C)\times G(D)$ is a sheaf.
\end{example}

Given the previously discussion about pseudo-pullbacks and Grothendieck pretopologies the next result may be seen as a corollary, but it incarnates our goal of obtaining a monoidal generalization of sheaves that provides usual sheaves (for a Grothendieck pretopology) whenever the monoidal tensor is the categorical product.

\begin{theorem}
Let $\mathcal{C} = (\mathcal{C},\times, 1)$ be a cartesian category with equalizers. If $F: \mathcal{C}^{op} \to Set$ is a sheaf for a Grothendieck pretopology, then $F$ is a sheaf for a Grothendieck prelopology. Conversely, a sheaf $F: \mathcal{C}^{op} \to Set$ for a  given Grothendieck prelopology is a sheaf for a  Grothendieck pretopology.
\end{theorem}
\begin{proof}
Assume $\mathcal{C} = (\mathcal{C},\times, 1)$ and consider a sheaf $F$ for a Grothendieck pretopology. By Proposition \ref{pro:everypretopology_is_a_prelopology}, every Grothendieck pretopology is a Grothendieck prelopology. Besides, in such conditions, the pullback in $\mathcal{C}$ is the pseudo-pullback in $\mathcal{C}$. Therefore, $F$ is a sheaf for a Grothendieck prelopology.
Conversely,  we use that the  pseudo-pullback in a cartesian category with equalizers is precisely the pullback, and that Grothendieck  prelopologies in cartesian categories with equalizers are Grothendieck 
 pretopologies, by Proposition \ref{prop:prelopology_in_cartesiancategory_is_pretology}. Thus, if $F: \mathcal{C}^{op} \to Set$ is a sheaf for a   Grothendieck prelopology then it is a sheaf for a  Grothendieck pretopology.
\end{proof}

This notion of a sheaf includes both our sheaves on quantales and the standard notion of sheaves for Grothendieck pretopologies. Besides, Example \ref{ex:sheaves_product_cat} says that if we have a sheaf $F$ for a Grothendieck pretopology and a sheaf $G$ on quantale, then we can obtain a sheaf $F \times G$ for the Grothendieck prelopology of the product category.

Now we show that under mild conditions we obtain that $F(U \otimes - )$ is a sheaf if $F$ is sheaf. This is crucial to show that the category of sheaves for a Grothendieck prelopology is monoidal closed, as we will see in the next section.

Let 
\begin{tikzcd}[ampersand replacement=\&]
	A \& B
	\arrow["g"', shift right, from=1-1, to=1-2]
	\arrow["f", shift left, from=1-1, to=1-2]
\end{tikzcd} be a pair of parallel arrows and $e : Eq(f,g) \to A$ be the equalizer arrow of such pair. We also have the equalizer $Eq(id_U \otimes f, id_U \otimes g) \to U \otimes A$, for any object $U$. Then there is a unique arrow $$\gamma : U \otimes Eq(f,g) \to Eq(id_U \otimes f, id_U \otimes g).$$

\begin{definition}
    We say that the functor $U \otimes - $ preserves equalizer when $\gamma$ is an isomorphism.
\end{definition}
\begin{remark}
    Analogously, we can say when $- \otimes U$  preserves equalizers. We say that the category is regular when $U \otimes - \otimes V$ preserves equalizers, for any $U, V$ objects in the category.
\end{remark}
In \cite[Examples 2.1.1]{aguiar1997internal}, there is a list of properties and examples of regular categories. In particular, the opposite of the category of associative algebras over a given base field with its usual tensor product is an example of semicartesian category that is regular (see \cite[Examples 2.1.1.5]{aguiar1997internal}).

\begin{lemma}\label{lem:commutewhereweneed}
Suppose that $(C,\otimes,1)$ is a semicartesian symmetric monoidal category and that $U\otimes -$ preserves equalizers.     There is a unique morphism $u$ such that the following diagram commutes
\[\begin{tikzcd}[ampersand replacement=\&]
	{U \otimes (V_i \tensor[_{f_i}]{\otimes}{_{f_j}} V_j) } \& {U \otimes (V_i \otimes V_j)} \\
	{(U\otimes V_i)\tensor[_{id_U \otimes f_i}]{\otimes}{_{id_U \otimes f_j}}  (U\otimes V_j) }
	\arrow["{id_U \otimes e}", from=1-1, to=1-2]
	\arrow["u", dashed, from=2-1, to=1-1]
	\arrow["{e' \circ (id_{U\otimes V_i}\otimes\pi^2_{U,V_j})\circ a_{U,V_i,V_j}}"'{pos=0.7}, from=2-1, to=1-2]
\end{tikzcd}\]
where $e'$ is the equalizer of form 
\begin{tikzcd}[ampersand replacement=\&,sep=scriptsize]
	{(U\otimes V_i)\tensor[_{id_U \otimes f_i}]{\otimes}{_{id_U \otimes f_j}}  (U\otimes V_j) } \& {(U\otimes V_i)\otimes (U\otimes V_j) }
	\arrow["{e'}", from=1-1, to=1-2]
\end{tikzcd} and $e$ is the equalizer of form 
\begin{tikzcd}[ampersand replacement=\&]
	{ V_i\tensor[_{f_i}]{\otimes}{_{f_j}}   V_j } \& { V_i\otimes V_j}
	\arrow["e", from=1-1, to=1-2]
\end{tikzcd}.
\end{lemma}
\begin{proof}

Since $U\otimes -$ preserves equalizers, we have that 
\[\begin{tikzcd}[ampersand replacement=\&]
	\&\& {U \otimes V_j} \\
	{U \otimes (V_i \tensor[_{f_i}]{\otimes}{_{f_j}} V_j)} \& {U \otimes (V_i \otimes V_j)} \&\& {U \otimes V} \\
	\&\& {U \otimes V_i}
	\arrow["{id_U \otimes f_j}", from=1-3, to=2-4]
	\arrow["{id_U \otimes f_i}"', from=3-3, to=2-4]
	\arrow["{id_U \otimes \pi^2_{V_i,V_j}}", from=2-2, to=1-3]
	\arrow["{id_U \otimes \pi^1_{V_i,V_j}}"'{pos=0.3}, from=2-2, to=3-3]
	\arrow["{id_U \otimes e}", from=2-1, to=2-2]
\end{tikzcd}\]
is an equalizer diagram.

By Proposition \ref{prop:equalizes}, the morphism $(U\otimes V_i) \tensor[_{id_U \otimes f_i}]{\otimes}{_{id_U \otimes f_j}} (U\otimes V_j) \to U \otimes (V_i \otimes V_j)$ given by the composition $e' \circ id_{U\otimes V_i}\otimes\pi^2_{U,V_j}\circ a_{U,V_i,V_j}$ equalizes  $(id_U \otimes f_j)\circ (id_U\otimes \pi^2_{V_i,V_j})$ and $(id_U \otimes f_i)\circ (id_U\otimes \pi^1_{V_i,V_j})$. 
So we obtain a unique $u: (U\otimes V_i)\tensor[_{id_U \otimes f_i}]{\otimes}{_{id_U \otimes f_j}}  (U\otimes V_j) \to U\otimes (V_i \otimes V_j)$ such that $e' \circ (id_{U\otimes V_i}\otimes\pi^2_{U,V_j})\circ a_{U,V_i,V_j} = (id_U \otimes e)\circ u$ by the universal property of the equalizer.
    
\end{proof}

\begin{proposition}\label{prop:Fotimesisasheaf}
    Suppose that $(C,\otimes,1)$ is a semicartesian symmetric monoidal category and that $U\otimes -$ preserves equalizers. If $F$ is a sheaf for a Grothendieck prelopology $L$ then $F(U \otimes -)$ is a sheaf for $L$.
\end{proposition}
\begin{proof}
    Take $\{x_{f_i} \in F(U \otimes V_i)\}$ a compatible family for $\{V_i \to V\}_{i\in I} \in L(V)$. So $F(p^1_{i,j})(x_{f_i}) = F(p^2_{i,j})(x_{f_j})$ in $F(U\otimes (V_i \tensor[_{f_i}]{\otimes}{_{f_j}} V_j))$, where $p^k_{i,j} =  (id_U \otimes\pi^k_{V_i,V_j})\circ (id_U \otimes e), $ for $k = 1,2$.  Because $F$ is a functor, using the unique arrow $u$ introduced in Lemma \ref{lem:commutewhereweneed}, we have $$F(p^1_{i,j}\circ  u)(x_{f_i}) = F(p^2_{i,j}\circ u)(x_{f_j}) \in F((U\otimes V_i)\tensor[_{id_U \otimes f_i}]{\otimes}{_{id_U \otimes f_j}}  (U\otimes V_j)) $$ 

    By the commutativity of the triangle in Lemma \ref{lem:commutewhereweneed}, we obtain that $F(p^1_{i,j}\circ  u)(x_{f_i}) = F(q)(x_{f_i})$ and $F(p^2_{i,j}\circ  u)(x_{f_j}) = F(q')(x_{f_j})$, where we call $q = id_U \otimes \pi^1_{V_i,V_j}\circ e' \circ (id_{U\otimes V_i}\otimes\pi^2_{U,V_j})\circ a_{U,V_i,V_j}$ and $q' = id_U \otimes \pi^2_{V_i,V_j}\circ e' \circ (id_{U\otimes V_i}\otimes\pi^2_{U,V_j})\circ a_{U,V_i,V_j}$.
    Thus we have $$F(q)(x_{f_i}) = F(q')(x_{f_j})\in F((U\otimes V_i)\tensor[_{id_U \otimes f_i}]{\otimes}{_{id_U \otimes f_j}}  (U\otimes V_j))$$

    In other words, $\{x_{f_i} \in F(U\otimes U_i)\}_{i\in I}$ is a compatible family for  $\{U \otimes V_i \to U \otimes V\}_{i\in I} \in L(U \otimes V)$.  Since $F$ is a sheaf, there is a unique $x \in F(U \otimes V)$ such that $F(id_U \otimes f_i)(x) = x_{f_i}$, for all $i \in I$. Therefore $F(U\otimes - )$ is a sheaf.
\end{proof}

The diagrammatic form of reading the above proof is the following: given $\{V_i \to V\}_{i\in I} \in L(V)$, $L$ is a Grothendieck prelopology we have  $\{U \otimes V_i \to U \otimes V\}_{i\in I} \in L(U\otimes V)$, for any object $U$ in $C$. Since $F$ is a sheaf, it is clear that 
\[\begin{tikzcd}[ampersand replacement=\&]
	{F(U\otimes V)} \& {\prod\limits_{i} F(U \otimes V_i)} \& {\prod\limits_{i,j} F(U \otimes ((U\otimes V_i)\tensor[_{id_U \otimes f_i}]{\otimes}{_{id_U \otimes f_j}}  (U\otimes V_j))}
	\arrow[from=1-1, to=1-2]
	\arrow[shift right, from=1-2, to=1-3]
	\arrow[shift left, from=1-2, to=1-3]
\end{tikzcd}\]
is an equalizer diagram.
Lemma \ref{lem:commutewhereweneed}, provides that the top diagram in 
\[\begin{tikzcd}[ampersand replacement=\&]
	{F(U\otimes V)} \& {\prod\limits_{i} F(U \otimes V_i)} \& {\prod\limits_{i,j} F(U \otimes (V_i \tensor[_{f_i}]{\otimes}{_{f_j}} V_j))} \\
	{F(U\otimes V)} \& {\prod\limits_{i} F(U \otimes V_i)} \& {\prod\limits_{i,j} F((U\otimes V_i)\tensor[_{id_U \otimes f_i}]{\otimes}{_{id_U \otimes f_j}}  (U\otimes V_j))}
	\arrow[from=1-1, to=1-2]
	\arrow[shift right, from=1-2, to=1-3]
	\arrow[shift left, from=1-2, to=1-3]
	\arrow["id"', from=1-1, to=2-1]
	\arrow["id"', from=1-2, to=2-2]
	\arrow[from=2-1, to=2-2]
	\arrow[shift right, from=2-2, to=2-3]
	\arrow[shift left, from=2-2, to=2-3]
	\arrow["{\prod F(u)}", from=1-3, to=2-3]
\end{tikzcd}\]
also is an equalizer. Therefore, $F(U \otimes -)$ is a sheaf. 

A close reading of the proof shows that we only used the third axiom in the definition of a Grothendieck prelopology. Apparently the fourth axiom have not played a role for this result (for now, that axiom was relevant only to show that Grothendieck prelopologies are precisely Grothendieck pretopologies if the category is cartesian). Moreover, if $F$ is a sheaf for a semicartesian quantale $Q$,  then $F(u \otimes -)$ is a sheaf for $Q$ without requiring $Q$ to be commutative. Thus, there are cases in which the category does not need to be symmetric\footnote{The symmetric is necessary to prove Proposition \ref{prop:equalizes}, used in the proof of Lemma \ref{lem:commutewhereweneed}.}. In the following we remove the hypothesis that $U \otimes -$ preserves equalizers by paying the price of adding another axiom in the notion of a Grothendieck pretopology. 

\begin{definition}
    Let $(C,\otimes,1)$ be a semicartesian monoidal category with pseudo-pullbacks. A \textbf{strong Grothendieck prelopology} on $C$ associates to each object $U$ of $C$ a set $L(U)$ of families of morphisms $\{U_i \rightarrow U \}_{i \in I}$ such that:
        \begin{enumerate}
        \item The singleton family $\{U' \xrightarrow{f} U \}$, formed by an isomorphism $f : U' \overset{\cong}\to U$,  is in $L(U)$;
        
        \item If $\{U_i \xrightarrow{f_i} U\}_{i \in I}$ is in $L(U)$ and $\{V_{ij} \xrightarrow{g_{ij}} U_i \}_{j \in J_i}$ is in $L(U_i)$ for all $i \in I$, then $\{V_{ij} \xrightarrow{f_i \circ g_{ij}} U \}_{i \in I, j \in J_i}$ is in $L(U)$; 
        \item If $\{ f_i : U_i \to U\}_{i \in I} \in L(U)$, then $\{ f_i \otimes id_V : U_i \otimes V \to U \otimes V\}_{i \in I}$ is in $L(U\otimes V)$ and $\{ id_V \otimes f_i : V \otimes U_i \to V \otimes U\}_{i \in I}$ is in $L(V\otimes U)$, for any $V$ object in $C$;
        \item If $\{U_i \xrightarrow{f_i} U\}_{i \in I}$ is in $L(U)$ and $g: V \to U$ is any morphism in $C$, then $\{\phi_i : U_i\tensor[_{f_i}]{\otimes}{_{g}}V \to U\tensor[_{id_U}]{\otimes}{_{g}}V\}_{i\in I}$ is in $ L(U\tensor[_{id_U}]{\otimes}{_{g}}V)$ and $\{\phi_i : V\tensor[_{g}]{\otimes}{_{f_i}}U \to V\tensor[_{g}]{\otimes}{_{id_U}}U\}_{i\in I}$ is in $ L(V\tensor[_{g}]{\otimes}{_{id_U}}U)$.
        \item If $\{U_i \xrightarrow{f_i} U\}_{i \in I}$ is in $L(U)$ and $V$ is any object in $C$, there is a morphism $l: (V \otimes U_i) \tensor[_{id_V \otimes f_i}]{\otimes}{_{id_V \otimes f_j}} V \otimes U_j \to V \otimes (U_i \tensor[_{f_i}]{\otimes}{_{f_j}} U_j)$ such that the following diagrams commute
\[\begin{tikzcd}[ampersand replacement=\&]
	{V \otimes (U_i \tensor[_{f_i}]{\otimes}{_{f_j}} U_j) } \&\& {V \otimes U_i} \\
	{(V\otimes U_i)\tensor[_{id_V \otimes f_i}]{\otimes}{_{id_V \otimes f_j}}  (V\otimes U_j) } \\
	{V \otimes (U_i \tensor[_{f_i}]{\otimes}{_{f_j}} U_j) } \&\& {V \otimes U_j} \\
	{(V\otimes U_i)\tensor[_{id_V \otimes f_i}]{\otimes}{_{id_V \otimes f_j}}  (V\otimes U_j) }
	\arrow["l", from=2-1, to=1-1]
	\arrow["{p^1_{V\otimes U_i,V\otimes U_j}}"'{pos=0.7}, from=2-1, to=1-3]
	\arrow["{id_V \otimes p^1_{U_i,U_j}}", from=1-1, to=1-3]
	\arrow["l", from=4-1, to=3-1]
	\arrow["{id_V \otimes p^2_{U_i,U_j}}", from=3-1, to=3-3]
	\arrow["{p^2_{V\otimes U_i,V\otimes U_j}}"', from=4-1, to=3-3]
\end{tikzcd}\]
and there is a morphism $r: (U_i \otimes V) \tensor[_{f_i \otimes id_V}]{\otimes}{_{f_j \otimes id_V}} U_j \otimes V \to  (U_i \tensor[_{f_i}]{\otimes}{_{f_j}} U_j)\otimes V $ such that the analogous diagrams commute.
        \end{enumerate}
\end{definition}\label{df:strongprelopolgy}

Observe that if $U \otimes -$ preserves equalizers then the unique morphism $u$ in Lemma \ref{lem:commutewhereweneed} is the morphism $l$ mentioned in the additional axiom. Observer that to prove Proposition \ref{prop:Fotimesisasheaf} we do not need the arrow $u$ to be unique neither we need $U \otimes -$ to preserve all equalizers.  Actually, we just need a good way to factorize $p^1_{V\otimes U_i,V\otimes U_j}$ and $p^2_{V\otimes U_i,V\otimes U_j}$. The fifth axiom is there precisely to provides such good factorization. In the case $C$ is a semicartesian quantale $Q$, the factorization holds because all morphism are unique and $v\odot u_i \odot v\odot u_j \leq v \odot u_i \odot  u_j \leq v \odot u_j, \forall v, u_i, u_j \in Q $. In the case $\mathcal{C}$ is cartesian category,  the universal property of the pullback provide a unique arrow allowing the factorization. Then we are still generalizing Grothendieck pretopologies in a way that encompass the quantalic case. Additionally, if the tensor is given by a weak product (in the sense that the universal arrow is not unique) then the pseudo-pullback is the weak pullback. Since the arrows in the factorization do not need to be unique, the fifth axiom is also satisfied in this context. 

It should be clear that we have the following proposition:
\begin{proposition}\label{prop:Fotimesisasheafstrong}
        Suppose that $(C,\otimes,1)$ is a semicartesian monoidal category. If $F$ is a sheaf for a strong Grothendieck prelopology $L$ then $F(U \otimes -)$ is a sheaf for $L$.
\end{proposition}
So in this way we do not to consider that $(C,\otimes,1)$ is symmetric.

\section{Sheafification}\label{sec:sheafification}

In classic sheaf theory, the full subcategory inclusion from sheaves (Grothendieck topos) to presheaves has a left adjoint functor that preserves finite limits.  The sheafification is that left adjoint functor. In the case of Grothendieck toposes, there are at least two ways to construct/find the sheafification. One of them consists of considering a semi-sheafification, also know as the \textit{plus construction}, that sends presheaves into separated presheaves by taking the colimit over all coverings (in the sense of a Grothendieck topology) for a fixed object in the category. If the presheaf already was separated, then the semi-sheafification gives a sheaf. So sheafification is the process of applying the semi-sheafification twice. Finally, it may be shown that this is left adjoint to the inclusion functor from sheaves to presheaves, and that the sheafification preserves finite limits \cite{maclane1992sheaves}. The idea of applying the semi-sheafification twice may be used in the case of Grothendieck pretopologies, as in \cite[Chapter 1]{MakkaiReyesFirstOrder}. Unfortunately, the obvious way to replicate it for Grothendieck prelopologies does not work. The first application of the semi-sheafification, even in the quantalic case, does not work: let $P$ be a  presheaf and $\{V_i\}_{i\in I}$ a covering of $V$. In the localic case, for each $V \leq U$ we define the map $P^+(U) \to P^+(V)$ by  $\{x_i \in P(U_i)\} \mapsto \eta_{\mathcal{V}}(\{x_{i|_{V\wedge U_i}}\in P(V)\})$, where $\eta_{\mathcal{V}}$ is the map from the set of compatibles families $Comp(\mathcal{V},P)$ to the colimit $P^+(V) = \underset{\mathcal{V} \in K(V)}{\varinjlim} Comp(\mathcal{V},P)$. Since $V\wedge U_i$ is a covering of $V$ whenever $U = \bigvee_{i \in I}U_i$, we have that $x_{i|_{V\wedge U_i}}\in P(V)$ and the semi-sheafification is a functor. However, in the quantalic case, $\odot$ is not idempotent and this implies that $V\odot U_i$ is \textbf{not} a covering of $V$ and then we are not able to define a map $P^+(U) \to P^+(V)$.

The other way to define the sheafification consists of looking at it in terms of local isomorphisms. A standard reference is \cite[Section 16]{Kashiwara_2006}. We will follow this approach, but we need to reintroduce sheaves. Recall that an $L$-cover is a cover in the sense of a Grothendieck prelopology.   

\begin{definition}\label{df:sieve}
Let $\{f_i : U_i \to U\}_{i \in I}$ be an $L$-cover of $U$. The \textbf{sieve} $S(\{U_i\})$ of $\{f_i : U_i \to U\}_{i \in I}$ is defined as the following coequalizer in $PSh(\mathcal{C})$:
\[\begin{tikzcd}
	{\coprod\limits_{i,j}y(u_i)\tensor[_{y(f_i)}]{\star}{}_{y(f_j)}y(u_j)} & {\coprod\limits_{i}y(u_i)} & {S(\{u_i\})} & {}
	\arrow[shift right=1, from=1-1, to=1-2]
	\arrow[from=1-2, to=1-3]
	\arrow[shift left=1, from=1-1, to=1-2]
\end{tikzcd}\]
where $\star$ is the Day convolution $PSh(\mathcal{C})$, $y$ is the Yoneda embedding, and the coproduct on the left is over the pseudo-pullbacks $y(U_j)\tensor[_{y(f_i)}]{\star}{}_{y(f_j)}y(U_j)$.
\end{definition}

\begin{remark}
This is the correspondent generalization of the sieve of a covering family in the sense of a Grothendieck pretopology.
\end{remark}

\begin{remark}
    The sieve $S(\{U_i\})$ of $\{f_i : U_i \to U\}_{i \in I} \in L(U)$ is a presheaf since it is a colimit of presheaves.
\end{remark}

We are not concerned, at this moment, with finding examples for this definition. All we want is to show that the inclusion $i : Sh(\mathcal{C}, L) \to PSh(\mathcal{C})$ has a left adjoint, and work with this abstract setting will allow us to do it. 

\begin{remark}
We have that 
\begin{tikzcd}[ampersand replacement=\&]
	{y(U_i)} \& {y(U)}
	\arrow[from=1-1, to=1-2]
\end{tikzcd} coequalizes 
\begin{tikzcd}[ampersand replacement=\&]
	{y(U_i)\tensor[_{y(f_i)}]{\star}{_{y(f_j)}}y(U_j)} \& {y(U_i)}
	\arrow[shift right, from=1-1, to=1-2]
	\arrow[shift left, from=1-1, to=1-2]
\end{tikzcd}. So, by the universal property of the coequalizer, we obtain a canonical morphism $$i_{\{U_i\}} : S(\{U_i\}) \to y(U) $$
for all $L$-cover $\{U_i\}_{i\in I}$ of $U.$
\end{remark}

With the above notion of sieves we can say that a presheaf $P$ is a sheaf if it is a local object with respect to all $i_{\{U_i\}}$. In other words: 

\begin{definition}\label{sheaf_as_local_object}
A \textbf{sheaf} in $(C,L)$ is a presheaf $P \in PSh(C)$ such that for all $L$-cover $\{f_i : U_i \to U\}_{i\in I}$ the hom-functor $Hom_{PSh(C)}(-,P)$ sends the canonical morphisms $i_{\{U_i\}} : S(\{U_i\}) \to y(U) $ to isomorphisms.
\[\begin{tikzcd}
	{Hom_{PSh(\mathcal{C})}(i_{\{U_i\}},P):Hom_{PSh(\mathcal{C})}(y(U),P)} & {Hom_{PSh(\mathcal{C})}(S(\{U_i\}),P)}
	\arrow["\cong", from=1-1, to=1-2]
\end{tikzcd}\]
\end{definition}
The above definition is saying that sheaves are \textit{local objects} with respect to the class of morphisms  $S(\{U_i\}) \to y(U)$.

Next, we want to show that when the base category $C$ is semicartesian and admits pseudo-pullbacks then the above definition coincides with the one we introduced before (Definition \ref{grothsheaf}). First we prove a useful lemma.

\begin{lemma}\label{lem:ppb_is_strongfunctor}
    If  $(\mathcal{C},\otimes, 1)$ is a  semicartesian category that admits pseudo-pullbacks then $$y(U_i\tensor[_{y(f_i)}]{\otimes}{_{y(f_j)}}U_j) \cong y(U_i)\tensor[_{y(f_i)}]{\star}{_{y(f_j)}} y(U_j)$$
\end{lemma}
\begin{proof}
Notice that applying the Yoneda embedding in the pseudo-pullback we have the following, since $y$ preserves limits:

\[\begin{tikzcd}
	{U_i\tensor[_{f_i}]{\otimes}{_{f_j}}U_j} &&& {y(U_i\tensor[_{y(f_i)}]{\otimes}{_{y(f_j)}}U_j)} \\
	& {U_i\otimes U_j} & {U_j} & {} & {y(U_i\otimes U_j)} & {y(U_j)} \\
	& {U_i} & U && {y(U_i)} & {y(U)}
	\arrow[from=2-3, to=3-3]
	\arrow[from=3-2, to=3-3]
	\arrow[from=2-2, to=2-3]
	\arrow[from=2-2, to=3-2]
	\arrow[from=1-1, to=2-2]
	\arrow[curve={height=12pt}, from=1-1, to=3-2]
	\arrow[shift left=1, curve={height=-12pt}, from=1-1, to=2-3]
	\arrow[from=2-5, to=2-6]
	\arrow[from=1-4, to=2-5]
	\arrow[from=2-5, to=3-5]
	\arrow[from=3-5, to=3-6]
	\arrow[from=2-6, to=3-6]
	\arrow[curve={height=12pt}, from=1-4, to=3-5]
	\arrow[curve={height=-18pt}, from=1-4, to=2-6]
	\arrow["y"{pos=0.3}, shorten <=23pt, shorten >=70pt, Rightarrow, from=2-3, to=2-5]
\end{tikzcd}\]

In other words, if $U_i\tensor[_{f_i}]{\otimes}{_{f_j}}U_j$ is the equalizer of the commutative square on the right, then $y(U_i\tensor[_{y(f_i)}]{\otimes}{_{y(f_j)}}U_j)$ is the equalizer of  the commutative square on the left.


Since the Yoneda embedding is a strong monoidal functor $(C,\otimes, 1) \to (PSh(\mathcal{C}),\star, y(1))$, we have $y(U_i \otimes U_j) \cong y(U_i)\star y(U_j)$. Then we have the following pseudo-pullback diagram
\[\begin{tikzcd}[ampersand replacement=\&]
	{y(U_i)\tensor[_{y(f_i)}]{\star}{_{y(f_j)}} y(U_j)} \\
	\& {y(U_i)\star y(U_j)} \& {y(U_j)} \\
	\& {y(U_i)} \& {y(U)}
	\arrow[from=2-2, to=2-3]
	\arrow[from=1-1, to=2-2]
	\arrow[from=2-2, to=3-2]
	\arrow[from=3-2, to=3-3]
	\arrow[from=2-3, to=3-3]
	\arrow[curve={height=12pt}, from=1-1, to=3-2]
	\arrow[curve={height=-18pt}, from=1-1, to=2-3]
\end{tikzcd}\]

Since the equalizer is unique, up to isomorphism, this implies that  $$y(U_i\tensor[_{y(f_i)}]{\otimes}{_{y(f_j)}}U_j) \cong y(U_i)\tensor[_{y(f_i)}]{\star}{_{y(f_j)}} y(U_j)$$
\end{proof}

 If the pseudo-pullback in the definition of a sieve is a pullback, then the $L$-cover is a Grothendieck pretopology and so we obtain a sheaf equipped with a Grothendieck pretopology. Moreover, under pullbacks, if $\{f_i: U_i \to U\}_{i \in I}$ is a Grothendieck pretopology covering of $U$ and $W$ is an object in $\mathcal{C}$, then $S(\{U_i\})(W)$ is described as the set of morphisms $h: W \to U$ such that each $h$ factors through one of the $U_i$. Indeed, since $S(\{U_i\})(W)$ is a coequalizer in $Set$ we have to describe the proper equivalence relation in $\coprod\limits_i y(U_i)$: given $\phi_i : W \to U_i$, we say that $\phi_i \sim \phi_j $ if there is $\phi_{ij}: W \to U_i \otimes_U U_j$ such that $p^1_{ij}\circ \phi_{ij} = \phi_i$ and $p^2_{ij}\circ \phi_{ij} = \phi_j$. When the pseudo-pullback $\otimes_U$ is a weak pullback $U_i \times_U U_j$ such equivalence relation is equivalent to saying that each $h : W \to U$ factors through $U_i$'s (and $W \to U_i \to U$ coincides with $W \to U_j \to U$ for all $i,j\in I$) because of the definition of the weak pullback, as we see below:
\[\begin{tikzcd}[ampersand replacement=\&]
	W \&\& {U_j} \\
	\& {U_i\times_U U_j} \\
	{U_i} \&\& U
	\arrow["{\phi_i}"', from=1-1, to=3-1]
	\arrow["{f_i}"', from=3-1, to=3-3]
	\arrow["{f_j}", from=1-3, to=3-3]
	\arrow["{p^1_{ij}}", from=2-2, to=3-1]
	\arrow["{p^2_{ij}}"', from=2-2, to=1-3]
	\arrow["h"{description}, curve={height=18pt}, from=1-1, to=3-3]
	\arrow["{\phi_j}", from=1-1, to=1-3]
	\arrow["{\phi_{ij}}", dashed, from=1-1, to=2-2]
\end{tikzcd}\]

However, the universal property of the pseudo-pullback (which is an equalizer) is not enough to give us that the existence of  $\phi_{ij}: W \to U_i \otimes_U U_j$ implies that $W \to U_i \to U$ and $W \to U_j \to U$ coincide. In other words, we do not have that the outer square in the following diagram commutes
\[\begin{tikzcd}[ampersand replacement=\&]
	W \&\& {U_j} \\
	\& {U_i\otimes_U U_j} \\
	{U_i} \&\& U
	\arrow["{\phi_i}"', from=1-1, to=3-1]
	\arrow["{f_i}"', from=3-1, to=3-3]
	\arrow["{f_j}", from=1-3, to=3-3]
	\arrow["{p^1_{ij}}", from=2-2, to=3-1]
	\arrow["{p^2_{ij}}"', from=2-2, to=1-3]
	\arrow["{\phi_j}", from=1-1, to=1-3]
	\arrow["{\phi_{ij}}", dashed, from=1-1, to=2-2]
\end{tikzcd}\]
The consequence is that in general the canonical arrow $i_{\{U_i\}} : S(\{U_i\}) \to y(U)$ is a monomorphism only when the pseudo-pullback is at least a weak pullback.

\begin{proposition}
If  $(\mathcal{C},\otimes,1)$ is a semicartesian category that admits pseudo-pullbacks, then the definition of sheaf as a local object (\ref{sheaf_as_local_object}) coincides with the first definition of a sheaf as a functor that makes a certain diagram an equalizer (\ref{df:sheaf_on_prelopology}).
\end{proposition}
\begin{proof}
If we apply $Hom_{PSh(\mathcal{C})}(-,P)$ in the coequalizer that defines the notion of sieves, we have
\[\begin{tikzcd}[ampersand replacement=\&,column sep=tiny]
	{Hom_{PSh(\mathcal{C})}(\coprod\limits_{i,j}y(U_i)\tensor[_{y(f_i)}]{\star}{}_{y(f_j)}y(U_j)},P) \& {Hom_{PSh(\mathcal{C})}(\coprod\limits_{i}y(U_i),P)} \& {Hom_{PSh(\mathcal{C})}(S(\{U_i\}),P)} 
	\arrow[shift right=1, from=1-1, to=1-2]
	\arrow[from=1-2, to=1-3]
	\arrow[shift left=1, from=1-1, to=1-2]
\end{tikzcd}\]
Since such $Hom_{PSh(\mathcal{C})}(-,P)$ sends colimits to limits, we have the following equalizer diagram
\[\begin{tikzcd}[ampersand replacement=\&,column sep=tiny]
	{Hom_{PSh(\mathcal{C})}(S(\{U_i\}),P)} \& {\prod\limits_{i}Hom_{PSh(\mathcal{C})}(y(U_i),P)} \& {\prod\limits_{i,j}Hom_{PSh(\mathcal{C})}(y(U_i)\tensor[_{y(f_i)}]{\star}{_{y(f_j)}} y(U_j),P)}
	\arrow[shift right=1, from=1-2, to=1-3]
	\arrow[shift left=1, from=1-2, to=1-3]
	\arrow[from=1-1, to=1-2]
\end{tikzcd}\]

Applying the Yoneda Lemma:

\[\begin{tikzcd}
	{Hom_{PSh(\mathcal{C})}(S(\{U_i\}),P)} & {\prod\limits_{i}P(U_i)} & {\prod\limits_{i,j}Hom_{PSh(\mathcal{C})}(y(U_i)\tensor[_{y(f_i)}]{\star}{_{y(f_j)}} y(U_j),P)}
	\arrow[shift right=1, from=1-2, to=1-3]
	\arrow[shift left=1, from=1-2, to=1-3]
	\arrow[from=1-1, to=1-2]
\end{tikzcd}\]
So, $P$ is a sheaf if and only if the following diagram is an equalizer (for each $L$-cover $\{U_i \to U\}_{i \in I}$)
\[\begin{tikzcd}
	{Hom_{PSh(\mathcal{C})}(y(U),P)} & {\prod\limits_{i}P(U_i)} & {\prod\limits_{i,j}Hom_{PSh(\mathcal{C})}(y(U_i)\tensor[_{y(f_i)}]{\star}{_{y(f_j)}} y(U_j),P)}
	\arrow[shift right=1, from=1-2, to=1-3]
	\arrow[shift left=1, from=1-2, to=1-3]
	\arrow[from=1-1, to=1-2]
\end{tikzcd}\]
Applying the Yoneda Lemma again, $P$ is a sheaf iff the following diagram is an equalizer
\[\begin{tikzcd}
	{P(U)} & {\prod\limits_{i}P(U_i)} & {\prod\limits_{i,j}Hom_{PSh(\mathcal{C})}(y(U_i)\tensor[_{y(f_i)}]{\star}{_{y(f_j)}} y(U_j),P)}
	\arrow[shift right=1, from=1-2, to=1-3]
	\arrow[shift left=1, from=1-2, to=1-3]
	\arrow[from=1-1, to=1-2]
\end{tikzcd}\]
Since $(C,L)$ is a $L$-site and $C$ is a category with pseudo-pullbacks, the pseudo-pullbacks $y(U_j)\tensor[_{y(f_i)}]{\star}{_{y(f_j)}} y(U_j)$ are representable functors. Then we apply Lemma \ref{lem:ppb_is_strongfunctor} and the Yoneda Lemma to obtain 
$$Hom_{PSh(\mathcal{C})}(y(U_j)\tensor[_{y(f_i)}]{\star}{_{y(f_j)}} y(U_j),P) \cong P(U_j \tensor[_{f_i}]{\otimes}{_{f_j}} U_j ) $$
So the sheaf condition is equivalent to requiring that 
\[\begin{tikzcd}
	{P(U)} & {\prod\limits_{i}P(U_i)} & {\prod\limits_{i,j} P(U_i \tensor[_{f_i}]{\otimes}{_{f_j}} U_j )}
	\arrow[shift right=1, from=1-2, to=1-3]
	\arrow[shift left=1, from=1-2, to=1-3]
	\arrow[from=1-1, to=1-2]
\end{tikzcd}\]
is an equalizer diagram for all coverings.
\end{proof}

\begin{definition}
A  \textbf{morphism of sheaves} is just a morphism of the underlying presheaves. 
\end{definition}

\begin{remark}
The category of sheaves $Sh(C,L)$ is a full subcategory of the category of presheaves $PSh(\mathcal{C})$.
\end{remark}

Now we recall some definitions as given in $1.32$ and $1.35$ of \cite{adamek1994locally}:

\begin{definition}
\begin{enumerate}
    \item An object $K$ is said to be \textbf{orthogonal} to a morphism $m: A \to A'$ provided that for each morphism $f: A \to K$ there exists a unique morphism $f': A' \to K$ such that the following triangle commutes
\[\begin{tikzcd}
	A && {A'} \\
	& K
	\arrow["m", from=1-1, to=1-3]
	\arrow["f"', from=1-1, to=2-2]
	\arrow["{f'}", from=1-3, to=2-2]
\end{tikzcd}\]
\item For each class $\mathcal{M}$ of morphisms in a category $\mathcal{K}$ we denote by $\mathcal{M}^{\perp}$ the full subcategory of $\mathcal{K}$ of all objects orthogonal to each $m:A\to A'$ in $\mathcal{M}.$
\end{enumerate}
\end{definition}

\begin{definition}
Let $\lambda$ be a regular cardinal. A $\lambda$\textbf{-orthogonality} class is a class of the form $\mathcal{M}^{\perp}$ such that every morphism in $\mathcal{M}$ has a $\lambda$-presentable\footnote{An object $C$ in a category $\mathcal{C}$ is $\lambda$-presentable, for $\lambda$ a regular cardinal, when the representable functor $Hom_{\mathcal{C}}(C,-)$ preserves $\lambda$-filtered limits.} domain and a $\lambda$-presentable codomain.
\end{definition}

\begin{theorem}\textbf{(Theorem $1.39$, \cite{adamek1994locally})}\label{teo1.39adamek}
Let $\mathcal{K}$ be a locally $\lambda$-presentable category. The following conditions on a full subcategory $\mathcal{A}$ of $\mathcal{K}$ are equivalent:
\begin{enumerate}
    \item[(i)] $\mathcal{A}$ is a $\lambda$-orthogonality class in $\mathcal{K}$;
    \item[(ii)] $\mathcal{A}$ is a reflective subcategory of $\mathcal{K}$ closed under $\lambda$-directed colimits.
\end{enumerate}
Furthermore, they imply that $\mathcal{A}$ is locally $\lambda$-presentable.
\end{theorem}

We know that $Sh(C,L)$ is a full subcategory of $PSh(\mathcal{C})$, and $PSh(\mathcal{C})$ is a $\lambda$-presentable category,  for every regular cardinal $\lambda$ that is sufficiently big. So, if we prove that $Sh(C,L)$ is a $\lambda$-orthogonality class in $PSh(\mathcal{C})$, we apply the above theorem to obtain that $Sh(C,L)$ is a reflective subcategory of $PSh(\mathcal{C})$. 

\begin{proposition}
$Sh(C,L)$ is a $\lambda$-orthogonality class in $PSh(\mathcal{C})$.
\end{proposition}
\begin{proof}
By definition, $F$ is a sheaf if and only if $$Hom_{PSh(\mathcal{C})}(i_{U_i},F):Hom_{PSh(\mathcal{C})}(y(U),F)\to Hom_{PSh(\mathcal{C})}(S(\{U_i\}),F) $$ is an isomorphism in $Set$, so $Hom_{PSh(\mathcal{C})}(i_{U_i},F)$ is a bijection. This means that for all $\varphi \in Hom_{PSh(\mathcal{C})}(S(\{U_i\}),F) $ there is a unique $\psi \in Hom_{PSh(\mathcal{C})}(y(U),F)$ such that $\psi \circ i_{U_i} = \varphi$. In other words, the desired triangle commutes:  

\[\begin{tikzcd}
	{S(\{U_i\})} && {y(U)} \\
	& F
	\arrow["{i_{U_i}}", from=1-1, to=1-3]
	\arrow["{\forall \varphi}"', from=1-1, to=2-2]
	\arrow["{\exists!\psi }", dashed, from=1-3, to=2-2]
\end{tikzcd}\]

Thus, all sheaves are orthogonal to $i_{U_i}.$
Then $Sh_L(C)  = \mathcal{M}^{\perp}$, where $\mathcal{M} = \{i_{U_i}: S(\{U_i\}) \to y(U) \,:\, \{U_i \to U\}_{i\in I} \in L(U)\}$ is a class of morphisms in $PSh(\mathcal{C})$. 

Since $C$ is a small category, the covering families of each $U$ are sets and we have a cardinal for each one of those sets. The supremum of all those cardinals is a cardinal again and then there is an even bigger regular cardinal $\lambda$ so that $PSh(\mathcal{C})$ is locally $\lambda$-presentable. Since $S(\{U_i\})$ and $y(U)$ are objects in $PSh(\mathcal{C})$, we conclude that $\mathcal{M}$ is a $\lambda$-orthogonality class in $PSh(\mathcal{C})$. 
\end{proof}
By definition of reflective subcategory:
\begin{corollary}
The inclusion functor $i: Sh(C,L) \to PSh(\mathcal{C})$ has a left adjoint functor $a : PSh(\mathcal{C}) \to Sh(C,L)$. Thus, $a$ preserves colimits.
\end{corollary}

Our sheafification look like the one already available in the literature but they cannot be the same.  It is known that Grothendieck toposes are the left exact reflective subcategories of a presheaf category. However, while $Sh(C,L)$  is a reflective subcategory of $PSh(C)$, in \cite[Theorem 2]{luiza2024sheaves} we proved that  $Sh(Q)$ is not always a Grothendieck topos and thus our sheafification cannot be left exact/preserve all finite limits. That is reasonable: since we are considering a weaker notion of covering (Grothendieck prelopologies instead of Grothendieck pretopologies), then we obtain a weaker sheafification that is a left adjoint functor of the inclusion but does not preserve all finite limits. 

The reader may wonder if we can present the sheafification through a formula. To prove Theorem \ref{teo1.39adamek}, the authors use an ``Orthogonal-reflection Construction'' that relies on  a transfinite induction. Thus, theoretically, we can find a formula to describe the sheafification using transfinite induction, but in practice it does not seem to have an elucidating structure. 

Next, note that Theorem \ref{teo1.39adamek} also provides that $Sh(C,L)$ is a locally $\lambda$-presentable category. We recall that:

\begin{definition}\cite[Definition 5.2.1]{borceux1994handbook2}
A category $\mathcal{M}$ is locally $\lambda$-presentable, for a regular cardinal $\lambda$, when
\begin{enumerate}
    \item $\mathcal{M}$ is cocomplete;
    \item $\mathcal{M}$ has a set $(G_i)_{i \in I}$ of strong generators;
    \item Each generator $G_i$ is $\lambda$-presentable. 
\end{enumerate}
\end{definition}

So, by Theorem \ref{teo1.39adamek}:

\begin{corollary}\label{cor:Shcocompleteandgenerators}
    $Sh(C,L)$ is cocomplete, has a set of strong generators with each one of then being $\lambda$-presentable.
\end{corollary}

Now, we will show that $a: PSh(C) \to Sh(C,L) $ preserves the monoidal closed structure of $PSh(C)$.
\begin{definition}\label{df:normal_enrichment}
    Let $\psi \dashv \phi: D \to B$ be an adjoint pair.
    \begin{enumerate}
        \item $\psi \dashv \phi$ is a  \textbf{reflective embedding} if $\phi$ is full and faithful on morphisms.
        \item When $B$ has a fixed monoidal closed structure the reflective embedding is called \textbf{normal} if there exists a monoidal closed structure on $D$ and monoidal functor structures on $\psi$ and $\phi$ for which $\phi$ is a normal closed functor and the unit and counit of the adjunction are monoidal natural transformations.
    \end{enumerate}
\end{definition}

The definition of ``closed functor'' is in  \cite{eilenberg1966closed} and the definition of ``normal closed functor'' is available from  \cite{barr1969adjunction}. We will not copy those definitions here since what is interesting for us is Day's observation \cite{day1973note} that normal enrichment is unique  (up to monoidal isomorphism) and it exists if and only if one condition of the following equivalent conditions is  satisfied:
\begin{align}
    \eta[b,\phi d] & \colon [b,\phi d] \cong \phi\psi[b,\phi d];\\
    [\eta , 1]  & \colon [\phi\psi b,\phi d] \cong [b,\phi d]; \\
    \psi(\eta \otimes 1 ) & \colon \psi(b \otimes b') \cong \psi(\phi\psi b \otimes b');\\
    \psi(\eta \otimes \eta ) & \colon \psi(b \otimes b') \cong  \psi(\phi\psi b \otimes \phi\psi b') \label{eq:monoidal_structure}
\end{align}
where $\eta$ is the unit of the adjunction $\psi \dashv \phi: D \to B$, $d$ is any object of $D$, $b$ is any object of $B$ and $[-,-]$ is the internal hom. In particular, the components $\Tilde{\phi}: \phi(b) \otimes_D \phi(b') \to \phi(b\otimes_B b')$ and $\phi^0: I_D \to \phi(I_B)$ are isomorphisms. 

Actually, the list in \cite{day1973note} is larger but we are using a shorter version that is more than enough for us and it is available in \href{https://ncatlab.org/nlab/show/Day%27s+reflection+theorem}{nLab}.


The next result is Proposition 1.1 in \cite{day1973note}, with a small notation change. 

\begin{proposition}\label{prop:3.1_day_monoidal_localization}
Let $C = (C,\otimes,1)$ be a small monoidal category and $S$ the cartesian closed category of small\footnote{The smallness condition for sets and for the monoidal category is to avoid size issues.} sets and set maps. Denote by $S^C$ the functor category from $C$ to $S$. A reflective embedding $\psi \dashv \phi: D \to S^C $ admits
normal enrichment if and only if the functor $F(U\otimes -)$ is isomorphic to some object in $D$ whenever $F$ is a object of $D$ and $U$ is an object of $C$.
\end{proposition}

\begin{theorem}
 Suppose that $(C,\otimes,1)$ is a semicartesian monoidal category 
 \begin{itemize}
     \item If $U \otimes -$ preserves equalizers and $(C,\otimes,1)$ is symmetric, then $Sh(C,L)$ is a monoidal closed category, where $L$ is a Grothendieck prelopology.
     \item Then $Sh(C,L)$ is a monoidal closed category, where $L$ is a strong Grothendieck prelopology.
 \end{itemize}
\end{theorem}
\begin{proof}
Considering the result from Day above, this is a consequence of Propositions \ref{prop:Fotimesisasheaf} and \ref{prop:Fotimesisasheafstrong}, respectively.
\end{proof}

Stated in other words:

\begin{proposition}
    The sheafification functor $a: PSh(\mathcal{C}) \to Sh(\mathcal{C},L)$ preserves the monoidal closed structure if $L$ is a Grothendieck prelopology and $U\otimes - $ preserves equalizers or if $L$ is a strong Grothendieck prelopology.
\end{proposition}

More precisely, since the normal enrichment gives that $\Tilde{a}: a(P) \otimes_{Sh(\mathcal{C})} a(P') \to a(P\otimes_{Day} P')$ is an isomorphism for any $P, P'$ presheaves and since $F \cong a(i(F))$, for every sheaf $F$, then the tensor of sheaves is $$F \otimes_{Sh(C)}  G \cong a(i(F)) \otimes_{Sh(C)}  a(i(G)) \cong  a(i(F) \otimes_{Day} i(G)).$$ 
\begin{remark}
    Note that $PSh(C,L)$ has cartesian products. In a certain way, our sheafification is ignoring the cartesian products and preserving the monoidal products that arise from the Day convolution.   
\end{remark}

The next result holds for $Sh(C,L)$, $L$ a Grothendieck prelopology, independently if the tensor functor preserves equalizers or not.

\begin{proposition}
   The functor $a: PSh(C) \to Sh(C,L)$ preserves terminal objects.
\end{proposition}
\begin{proof}
   Since the terminal presheaf is a sheaf  the fact that the inclusion of sheaves into presheaves is a fully faithful functor and thus reflects all limits imples that  the  terminal sheaf is the terminal presheaf, which is the constant presheaf with value the terminal object in $Set$.
\end{proof}

The construction used to define the sheafification also leads to other properties of $Sh(\mathcal{C},L)$ that makes it look like as a Grothendieck topos. Since $Sh(\mathcal{C},L)$ is complete (the limits are computed point-wise) and has a set of strong generators (Corollary \ref{cor:Shcocompleteandgenerators}), we have that $Sh(\mathcal{C},L)$ is well-powered \cite[Proposition 4.5.15]{borceux1994handbook}. Then, for any $F$ sheaf in $Sh(\mathcal{C},L)$, we have that the poset of subobjects $Sub(F)$ has all infima and suprema \cite[Corollary  4.2.5]{borceux1994handbook}. It is known that $Sub(F)$ is a locale when $F$ is a sheaf for a Grothendieck topology. We prove in \cite{luiza2024sheaves}, that if $F$ is a sheaf on a quantale $Q$ as in Definition \ref{df:sheaf-on-quantales} then $Sub(F)$ is isomorphic to the quantale $Q$. Therefore, if binary infima does not distributes over arbitrary suprema in $Q$, the $Sub(F)$ is not a locale and so $Sh(Q)$ is not a Grothendieck topos. Actually, is not even a topos.

If $F$ sheaf in $Sh(\mathcal{C},L)$ the next problem to solve is: $(i)$ does $Sub(F)$ have a associative binary operation $*$? $(ii)$ if so, does $*$ distributes over the arbitrary suprema in $Sub(F)$? In the following we show a path to construct $*$.

\begin{proposition}
    {\bf Factorization of morphisms in} $Sh(\mathcal{C},L)$:

For each morphism $\phi: F \to G$ in $Sh(\mathcal{C},L)$, there exists the least subobject of $G$, represented by $\iota : G' \rightarrowtail G$, such that $\phi = \iota \circ \phi'$ for some (and thus, unique) morphism $\phi' : F \to G'$. Moreover, $\phi'$ is an epimorphism.
\end{proposition}

\begin{proof}
By the previous discussion, there exists the  extremal factorization above $\phi = \iota \circ \phi'$, such that $\iota : G' \rightarrowtail G$ is a mono. To show that $\phi' : F \to G'$ is an epi, consider $\eta, \epsilon: G' \to H$ such that $\eta \circ \phi' =  \epsilon \circ \phi'$ and let $\gamma  : H' \rightarrowtail G'$ be the equalizer of $\eta, \epsilon$. Then, by the universal property of $\gamma$, there exists a unique $\phi'' : F \to H'$ such that $\gamma \circ \phi'' = \phi'$. On the other hand, by the extremality of $\iota$, there exists a unique $\gamma ' : G' \to H'$ such that $\iota = \iota \circ \gamma \circ \gamma'$. Since $\iota$ is a mono, we obtain that $\gamma \circ \gamma' = id_{G'}$. Thus $\gamma$ is a mono that is a retraction. This means that $\gamma = eq(\eta, \epsilon)$ is an iso, i.e., $\epsilon = \eta$. Then $\phi'$ is an epi.
\end{proof}
\begin{remark}
    The above proof is precisely the same we used to prove the same result but for $Sh(Q)$ in \cite{luiza2024sheaves}. The next definition was also used in that same paper.
\end{remark}

\begin{definition}
    For each $F$ sheaf on $Q$, we define the following binary operation on $Sub(F)$: Given $\phi_i : F_i \rightarrowtail F$, $i =0,1$ define $\phi_0 * \phi_1 : F_0 * F_1  \rightarrowtail F$ as the mono in the extremal factorization of $F_0 \tensor[_{\phi_0}]{\otimes}{_{\phi_1}} F_1 \rightarrowtail F_0 \otimes F_1 \rightrightarrows F$.
\end{definition}

In the quantalic case this construction permitted us to show that $Sub(F)$ is indeed a quantale. Unfortunately, the argument used in \cite{luiza2024sheaves} is not easily translated to the present more general scenario. We leave further investigations regarding this for a future work.

\section{Towards Grothendieck loposes and conclusions}
Given our definition of a sheaf for a Grothendieck prelopology, the reader may expect a definition of a Grothendieck lopology so that we define sheaves for it and then define a Grothendieck lopos as a category that is equivalent to a category of sheaves for a Grothendieck lopology. However, in general, the canonical arrow $i_{\{U_i\}} : S(\{U_i\}) \to y(U)$ is not a monomorphism, when $\{f_i : U_i \to U\}$ is an $L$-cover of $U$. Thus, a sieve on $U$ is not necessarily a subfunctor of the Yoneda embedding $y(U)$, which is always the case in the context of Grothendieck pretopologies. 

Fortunately, there are many alternative characterizations of Grothendieck toposes. Here we recall that Grothendieck toposes are categorizations of locales. One result to see this point of view states that $\mathcal{C}$ is a Grothendieck topos if and only if $\mathcal{C}$ is lex total and has a small set of generators \cite{street1981notions}, where lex total means that the Yoneda $y: \mathcal{C} \to Set^{\mathcal{C}^{op}}$ has a left adjoint functor that preserves all finite limits. This result is the categorization of the fact that a poset $P$ is locale if and only if the Yoneda  $y: P \to 2^{P^{op}}$ ($2 = \{0\leq 1\}$) has a left adjoint that preserves finite infima. Indeed: given a poset $P$, note that the category $2^{P^{op}}$ consists of ordered functions $P^{op} \to 2$ where the inverse image of $1$ gives a down-set $D$ of $P$, that is, if $a\leq b$ and $b \in D$, then $a \in D$. Thus, the Yoneda $y(a)$ corresponds to the down-set $\downarrow a = \{x \in P \,:\, x \leq a\}$. We also have that $y: P \to 2^{P^{op}}$ has left adjoint if and only if $P$ has arbitrary sups. Indeed, the left adjoint is $\bigvee : 2^{P^{op}} \to P$, which takes a down-set $A$ and send it to $\bigvee A$. The adjoint condition $A \subseteq \downarrow a \iff \bigvee A \leq a$ holds by definition of supremum. Now, it should be clear that

\begin{proposition}
    Suppose that $P$ has arbitrary sups and a binary associative operation $\odot : P \times P \to P$. Then $P$ is a quantale if and only if $\bigvee : 2^{P^{op}} \to P$ preserves $\odot$.
\end{proposition}

In other words, a poset $P$ with an associative binary operation $\odot$ is a quantale if and only if the Yoneda  $y: P \to 2^{P^{op}}$ has a left adjoint that preserves $\odot$.
In the localic case,  $\bigvee : 2^{P^{op}} \to P$ preserves $\wedge$ categorifies to the left adjoint of the Yoneda preserves all finite limits because in a locale the meet is both the product and the pullback in the category. In the quantalic case, the multiplication is both the monoidal product and the pseudo-pullback. Therefore, we suggest that a category  $\mathcal{C}$ should be a Grothendieck lopos if and only if $\mathcal{C}$ has a small set of generators and $y: \mathcal{C} \to Set^{\mathcal{C}^{op}}$ has a left adjoint monoidal functor that preserves pseudo-pullback, where saying that a functor $\Lambda$ preserves pseudo-pullback means that if $(A \otimes_C B)$ is a pseudo-pullback then $F(A \otimes_C B)$ is isomorphic to the pseudo-pullback $F(A) \otimes_{F(C)}' F(B)$, using the respective tensors in the considered categories.

This approach highlights that we aim to a categorization of quantales in the same way that Grothendieck toposes are a  categorization of locales. We are aware of other characterization of Grothendieck toposes that we will try to replicate in future works, the most famous one being the Giraud Theorem.

Finally, we say that the axioms for Grothendieck prelopologies (including the weak and the strong versions) may be improved. It is clear that the third axiom is crucial to obtain that $Sh(C,L)$ is a closed monoidal category but it alone is not enough, that is why we add the fifty axiom (or, alternatively, restrict or attention to symmetric categories $C$ and ask $U \otimes - $ to preserve equalizers). The role of the fourth axiom is to provide that a sheaf for a Grothendieck prelopology is exactly the a sheaf for a Grothendieck pretopology if the monoidal structure in $C$ is cartesian. We want this to keep the analogy with the posetal case: our sheaves on quantales are precisely sheaves on locales if the multiplication is the meet operation. However, we believe that the sheafification should preserve more structure -- equalizers or at least pseudo-pullbacks -- and such property must come from the axioms of the prelopology. Checking this will be the goal of future works.

\section{Acknowledgments}
This paper contains a significant part of the thesis of Ana Luiza Tenório, which was funded by Coordenação de Aperfeiçoamento de Pessoal de Nível Superior (CAPES). Grant Number 88882.377949/2019-01.

\section{Appendix: The monoidal projection's behavior}\label{secA1}

Here we need to be careful with the notation of our projections. We use $\pi^1_{Y,Z} = \rho_Y \circ (id_Y \otimes !_Z)$ and $\pi^2_{Y,Z} = \lambda_Z \circ (!_Y \otimes id_Z)$ to denote, respectively, the projections 
\begin{center}
 $\pi^1_{Y,Z}:  Y \otimes Z \to Y$ and $\pi^2_{Y,Z}:  Y \otimes Z \to Z$,   
\end{center}
for any objects $Y,Z$. We want to show the following:

\begin{proposition}\label{prop:equalizes}
    Let $a_{X,A,B}: (X\otimes A)\otimes B \to X \otimes (A \otimes B)$ be the associator, and $e : (X \otimes A) \tensor[_{\phi}]{\otimes}{_{\psi}} (X \otimes B) \to (X \otimes A) \otimes (X \otimes B)$ the equalizer in the pseudo-pullback of $\phi = (id_X \otimes f) \circ \pi^1_{X\otimes A, X\otimes B}$ and $\psi =  (id_X \otimes g) \circ \pi^2_{X\otimes A, X\otimes B}$, where $f: A \to C$ and $g: B \to C$. Then the composition $a_{X,A,B} \circ (id_{X\otimes A} \otimes \pi^2_{X,B})\circ e$ equalizes the pair $(id_X \otimes f) \circ (id_X \otimes \pi^1_{A,B})$ and  $(id_X \otimes g) \circ (id_X \otimes \pi^2_{A,B})$.
\end{proposition}

This result may be expected but it does requires a verification since our projections are not as well-behaved as the projections of the categorical product. First, we will check minor lemmas regarding those projections. In particular, how they interact with the braiding in a symmetric monoidal category.

We start recalling a proposition from \cite{etingof2016tensor}.

\begin{proposition} \label{prop:consequence2.2.4} (Propostion 2.2.4 \cite{etingof2016tensor})
    The following diagrams commute for all objects $A, B$ in a monoidal category.
\end{proposition}
\begin{enumerate}
    \item 
\[\begin{tikzcd}[ampersand replacement=\&]
	{(1 \otimes A) \otimes B} \&\& {1 \otimes (A \otimes B)} \\
	\& {A\otimes B}
	\arrow["{a_{1,A,B}}", from=1-1, to=1-3]
	\arrow["{\lambda_A \otimes id_B}"', from=1-1, to=2-2]
	\arrow["{\lambda_{A\otimes B}}", from=1-3, to=2-2]
\end{tikzcd}\]
    \item 
\[\begin{tikzcd}[ampersand replacement=\&]
	{(A \otimes B) \otimes 1} \&\& {A \otimes (B \otimes 1)} \\
	\& {A\otimes B}
	\arrow["{a_{A,B,1}}", from=1-1, to=1-3]
	\arrow["{\rho_{A\otimes B}}"', from=1-1, to=2-2]
	\arrow["{id_A \otimes \rho_B}", from=1-3, to=2-2]
\end{tikzcd}\]
\end{enumerate}

From this is immediate that 
\begin{lemma}\label{lem:factoringminorprojections}
The following diagrams commute for all objects $A, B, X$ in a semicartesian monoidal category:
\begin{align}
\begin{tikzcd}[ampersand replacement=\&]\label{diagram:p2}
	{(X \otimes A) \otimes B} \&\& {X \otimes (A \otimes B)} \\
	\& {A\otimes B}
	\arrow["{a_{X,A,B}}", from=1-1, to=1-3]
	\arrow["{\pi^2_{X,A}\otimes id_B}"', from=1-1, to=2-2]
	\arrow["{\pi^2_{X,A\otimes B}}", from=1-3, to=2-2]
\end{tikzcd}
\end{align}
\begin{align}\label{diagram:p1}
\begin{tikzcd}[ampersand replacement=\&]
	{(A \otimes B) \otimes X} \&\& {A \otimes (B \otimes X)} \\
	\& {A\otimes B}
	\arrow["{a_{A,B,X}}", from=1-1, to=1-3]
	\arrow["{\pi^1_{A\otimes B,X}}"', from=1-1, to=2-2]
	\arrow["{id_A \otimes\pi^1_{B,X}}", from=1-3, to=2-2]
\end{tikzcd}
\end{align}
\end{lemma}
\begin{proof}
    The fisrt diagram commutes because in 
\[\begin{tikzcd}[ampersand replacement=\&]
	{(X \otimes A) \otimes B} \&\& {X \otimes (A \otimes B)} \\
	{(1\otimes A)\otimes B} \&\& {1 \otimes (A \otimes B)} \\
	\& {A\otimes B}
	\arrow["{a_{X,A,B}}", from=1-1, to=1-3]
	\arrow["{(!_X \otimes id_A)\otimes id_B}"', from=1-1, to=2-1]
	\arrow["{!_X \otimes id_{A\otimes B}}", from=1-3, to=2-3]
	\arrow["{\lambda_A \otimes id_B}"', from=2-1, to=3-2]
	\arrow["{\lambda_{A\otimes B}}", from=2-3, to=3-2]
	\arrow["{a_{1,A,B}}", from=2-1, to=2-3]
\end{tikzcd}\]

we have that the square commutes because of the naturality of the associator and the triangle commutes because of the above proposition.
Analogously, the diagram $A2$ commutes.
\end{proof}

While those diagrams shows a way of deleting objects on the extremes (the first or the last component), we can also delete objects in the middle: the definition of a monoidal category gives that 
\[\begin{tikzcd}[ampersand replacement=\&]
	{(A\otimes 1)\otimes B} \&\& {A \otimes (1 \otimes B)} \\
	\& {A\otimes B}
	\arrow["{\rho_A \otimes id_B}"', from=1-1, to=2-2]
	\arrow["{id_A \otimes \lambda_{B}}", from=1-3, to=2-2]
	\arrow["{a_{A,1,B}}", from=1-1, to=1-3]
\end{tikzcd}\]
commutes. Using the same reasoning as before, we obtain the following: 

\begin{lemma}
The following diagram commutes for all objects $A, B, X$ in a semicartesian monoidal category:
\begin{align}\label{diagram:p1andp2}
\begin{tikzcd}[ampersand replacement=\&]
	{(A\otimes X)\otimes B} \&\& {A \otimes (X \otimes B)} \\
	\& {A\otimes B}
	\arrow["{\pi^1_{A,X}\otimes id_B}"', from=1-1, to=2-2]
	\arrow["{id_A \otimes \pi^2_{X,B}}", from=1-3, to=2-2]
	\arrow["{a_{A,X,B}}", from=1-1, to=1-3]
\end{tikzcd}
\end{align}
\end{lemma}

\begin{lemma}\label{lem:factoringprojections}
The following diagrams commute for all objects $A, B, X$ in a semicartesian monoidal category:
\[\begin{tikzcd}[ampersand replacement=\&]
	\&\& {(X\otimes A)\otimes(X\otimes B)} \&\& {(X\otimes A)\otimes(X\otimes B)} \\
	\&\& {(X\otimes A)\otimes B} \\
	{X\otimes A} \&\& {X\otimes (A\otimes B)} \&\& {A\otimes (X\otimes B)} \& {X\otimes B}
	\arrow["{\pi^1_{X\otimes A,X\otimes B}}"', from=1-3, to=3-1]
	\arrow["{id_X \otimes \pi^1_{A,B}}", from=3-3, to=3-1]
	\arrow["{id_{X \otimes A}\otimes \pi^2_{X,B}}", from=1-3, to=2-3]
	\arrow["{a_{X,A,B}}", from=2-3, to=3-3]
	\arrow["{\pi^2_{X,A}\otimes id_{X\otimes B}}"', from=1-5, to=3-5]
	\arrow["{\pi^2_{X\otimes A,X\otimes B}}", from=1-5, to=3-6]
	\arrow["{\pi^2_{A, X\otimes B}}"', from=3-5, to=3-6]
\end{tikzcd}\]
\end{lemma}
\begin{proof}
    For the second diagram observe that $!_{X\otimes A} = !_A \circ \lambda_A \circ (!_X \otimes id_A)$ because the terminal arrow is unique. So, by functoriality of $\otimes$: 
    \begin{align*}
        \pi^2_{X\otimes A\otimes X\otimes B} &= \lambda_{X\otimes B} \circ (!_{X\otimes A} \otimes id_{X\otimes B}) \\
        &= \lambda_{X\otimes B} \circ ( (!_A \circ \lambda_A \circ (!_X \otimes id_A)) \otimes id_{X\otimes B})\\
        &= \lambda_{X\otimes B} \circ  (!_A \otimes id_{X\otimes B}) \circ (\lambda_A \circ (!_X \otimes id_A) \otimes id_{X\otimes B}) \\
        &=\lambda_{X\otimes B} \circ  (!_A \otimes id_{X\otimes B}) \circ (\pi^2_{X,A} \otimes id_{X\otimes B})\\
        &=\pi^2_{A,X\otimes B} \circ (\pi^2_{X,A} \otimes id_{X\otimes B})
    \end{align*}

For the first diagram we also consider the naturality of $a$ and Proposition \ref{prop:consequence2.2.4}.
    \begin{align*}
        \pi^1_{X\otimes A\otimes X\otimes B} &= \rho_{X\otimes B} \circ (id_{X\otimes A} \circ !_{X\otimes B}) \\
        &= \rho_{X\otimes A} \circ (id_{X\otimes A} \otimes (!_B \circ \lambda_B \circ (!_X \otimes id_B)))\\
        &= \rho_{X\otimes A} \circ (id_{X\otimes A}\otimes!_B  ) \circ (id_{X\otimes A}\otimes(\lambda_B \circ (!_X \otimes id_B))) \\
        &= \rho_{X\otimes A} \circ (id_{X\otimes A}\otimes!_B  ) \circ (id_{X\otimes A}\otimes \pi^2_{X,B})\\
        &= (id_X \otimes \rho_A) \circ a_{X,B,1} \circ (id_{X\otimes A}\otimes!_B  ) \circ (id_{X\otimes A}\otimes \pi^2_{X,B})\\
        &= (id_X \otimes \rho_A) \circ (id_X \otimes (id_A \otimes !_B) \circ a_{X,A,B}) \circ (id_{X\otimes A}\otimes \pi^2_{X,B})\\
        &= id_X \otimes (\rho_A \circ (id_A \otimes !_B)) \circ a_{X,A,B}) \circ (id_{X\otimes A}\otimes \pi^2_{X,B})\\
        &= (id_X \otimes \pi^1_{A,B}) \circ a_{X,A,B} \circ (id_{X\otimes A}\otimes \pi^2_{X,B}).
    \end{align*}

\end{proof}

Now we are gonna use the braiding. We start using an already known result, stated in \cite{joyal1993braided} with a detailed proof in \cite{johnson2021bimonoidal};
\begin{proposition} (Proposition 1.3.21 \cite{johnson2021bimonoidal})
         The following diagrams commute for all objects $A$ in a braided monoidal category:
\[\begin{tikzcd}[ampersand replacement=\&]
	{A\otimes 1} \&\& {1 \otimes A} \& {1\otimes A} \&\& {A\otimes 1} \\
	\& A \&\&\& A
	\arrow["{\rho_A}"', from=1-1, to=2-2]
	\arrow["{\lambda_A}", from=1-3, to=2-2]
	\arrow["{b_{A,1}}", from=1-1, to=1-3]
	\arrow["{b_{1,A}}", from=1-4, to=1-6]
	\arrow["{\rho_A}", from=1-6, to=2-5]
	\arrow["{\lambda_A}"', from=1-4, to=2-5]
\end{tikzcd}\]
\end{proposition}

Using the naturality of the braiding we have that $(!_B \otimes id_A) \circ b_{A,B} = b_{A,1} \circ (id_A \otimes !_B)$. Thus
     
\begin{lemma}\label{lem:projectionsandbraiding}
The following diagrams commute for all objects $A,B$ in a braided semicartesian monoidal category:
\[\begin{tikzcd}[ampersand replacement=\&]
	{ A\otimes B} \&\& {B\otimes A} \& { A\otimes B} \&\& {B\otimes A} \\
	\& A \&\&\& B
	\arrow["{b_{A,B}}", from=1-1, to=1-3]
	\arrow["{b_{A,B}}", from=1-4, to=1-6]
	\arrow["{\pi^2_{A,B}}"', from=1-4, to=2-5]
	\arrow["{\pi^1_{B,A}}", from=1-6, to=2-5]
	\arrow["{\pi^2_{B,A}}", from=1-3, to=2-2]
	\arrow["{\pi^1_{A,B}}"', from=1-1, to=2-2]
\end{tikzcd}\]
\end{lemma}

Now we are ready to prove Proposition\ref{prop:equalizes}:
\begin{proof}
    We will show that  $(id_X \otimes f) \circ (id_X \otimes \pi^1_{A,B}) \circ a_{X,A,B} \circ (id_{X\otimes A} \otimes \pi^2_{X,B})\circ e = (id_X \otimes g) \circ (id_X \otimes \pi^2_{A,B}) \circ a_{X,A,B} \circ (id_{X\otimes A} \otimes \pi^2_{X,B}) \circ e$.
\end{proof}

By Lemma \ref{lem:factoringprojections}, we know that 
$$ \pi^1_{X\otimes A\otimes X\otimes B}  = (id_X \otimes \pi^1_{A,B}) \circ a_{X,A,B} \circ (id_{X\otimes A}\otimes \pi^2_{X,B}).$$
By hypothesis, $e$ equalizes $(id_X \otimes f) \circ \pi^1_{X\otimes A, X \otimes B}$
and $(id_X \otimes g) \circ \pi^2_{X\otimes A, X \otimes B}$. Then
{\small 
  \setlength{\abovedisplayskip}{6pt}
  \setlength{\belowdisplayskip}{\abovedisplayskip}
  \setlength{\abovedisplayshortskip}{0pt}
  \setlength{\belowdisplayshortskip}{3pt}
\begin{align*}
   (id_X \otimes f) \circ  \pi^1_{X\otimes A\otimes X\otimes B} \circ e &=  (id_X \otimes g) \circ\pi^2_{X\otimes A, X \otimes B} \circ e \\
   &=  (id_X \otimes g) \circ \pi^2_{A,X\otimes B} \circ (\pi^2_{X,A} \otimes id_{X\otimes B}) \circ e & (\ref{lem:factoringprojections})\\
   &= (id_X \otimes g) \circ \pi^2_{A,X\otimes B} \circ (\pi^1_{A,X} \otimes id_{X\otimes B})\circ(b_{A,X}\otimes id_{X\otimes B})^{-1}\circ e  &(\ref{lem:projectionsandbraiding})\\
\end{align*}
}%

So it only remains to show that $$ \pi^2_{A,X\otimes B} \circ (\pi^1_{A,X} \otimes id_{X\otimes B})\circ(b_{A,X}\otimes id_{X\otimes B})^{-1} = (id_X \otimes \pi^2_{A,B}) \circ a_{X,A,B} \circ (id_{X\otimes A} \otimes \pi^2_{X,B}).$$

To see this we apply the above Lemmas as follows:
\begin{align*}
    &(id_X \otimes \pi^2_{A,B}) \circ a_{X,A,B} \circ (id_{X\otimes A} \otimes \pi^2_{X,B}) \\ &=  (\pi^1_{X,A} \otimes id_B)\circ a_{X,A,B}^{-1}\circ a_{X,A,B} \circ (id_{X\otimes A} \otimes \pi^2_{X,B}) & \ref{diagram:p1andp2}\\
    &= (\pi^1_{X,A} \otimes id_B)\circ (id_{X\otimes A} \otimes \pi^2_{X,B})\\
    &= (\pi^2_{A,X}\otimes id_B)\circ (b_{X,A} \otimes id_B) \circ (id_{X\otimes A} \otimes \pi^2_{X,B}) & \ref{lem:projectionsandbraiding}\\
    &= \pi^2_{A, X\otimes B} \circ a_{A,X,B} \circ (b_{X,A} \otimes id_B) \circ (id_{X\otimes A} \otimes \pi^2_{X,B}) & \ref{diagram:p2} \\
    &= \pi^2_{A, X\otimes B} \circ a_{A,X,B} \circ (id_{A\otimes X} \otimes \pi^2_{X,B})\circ (b_{X,A} \otimes id_{X \otimes B}) & \mbox{Naturality} \\
    &= \pi^2_{A, X\otimes B} \circ a_{A,X,B} \circ (\pi^1_{A\otimes X, X}) \circ a^{-1}_{A\otimes X,X,B} \circ  (b_{X,A} \otimes id_{X \otimes B}) & \ref{diagram:p1andp2} \\
    &= \pi^2_{A, X\otimes B} \circ (\pi^1_{A,X \otimes id_{X\otimes B}})\circ a_{A\otimes X,X,B} \circ a^{-1}_{A\otimes X,X,B} \circ  (b_{X,A} \otimes id_{X \otimes B})  & \mbox{Naturality} \\
    &= \pi^2_{A, X\otimes B} \circ (\pi^1_{A,X \otimes id_{X\otimes B}}) \circ  (b_{X,A} \otimes id_{X \otimes B}).
\end{align*}
Since the category is symmetric, $b_{X,A} = b_{A,X}^{-1}$. Thus the desired equality holds.

\bibliographystyle{abbrv} 
\bibliography{ArXiv_sh_q}

\end{document}